\documentclass[11pt,oneside]{article}
\usepackage{times,amsmath,amsbsy,amssymb,amscd,mathrsfs,graphicx}
\usepackage{amssymb}
\usepackage{amsmath,bm}
\usepackage{amsmath,amsthm}
\usepackage{framed}
\usepackage{xcolor,color}
\usepackage[font=small,labelfont=md,textfont=it]{caption}
\usepackage{floatrow,float}
\usepackage[titletoc, title]{appendix}
\usepackage[colorlinks,linkcolor=blue,citecolor=blue]{hyperref}
\usepackage{longtable}
\usepackage{pgfplots}
\usepackage{diagbox}
\usepackage{booktabs,makecell,multirow}
\usepackage[capitalise, nosort]{cleveref}
\usepackage{algorithm}
\usepackage{algorithmic}
\usepackage{amsmath,bm}
\usepackage{cases}
\usepackage{geometry}
\usepackage{extarrows}
\usepackage{tcolorbox}
\usepackage{bbm}
\usepackage{syntonly}
\usepackage[figuresright]{rotating}
\usepackage{epstopdf}
\usepackage{graphicx}
\usepackage{tikz}
\tikzset{elegant/.style={smooth,thick,samples=50,black}}
\tikzset{eaxis/.style={->,>=stealth}}

\usepackage[misc]{ifsym}

\crefname{equation}{}{}
\crefname{lem}{Lemma}{Lemmas}
\crefname{thm}{Theorem}{Theorems}
\crefname{assum}{Assumption}{Assumptions}

\newcommand{\proj}[0]{ {\bf proj}}

\newcommand{\conv}[1]{{\bf conv}\left\{ {#1} \right\}}

\newcommand{\dd}{\,{\rm d}}

\newcommand{\R}{\,{\mathbb R}}

\newcommand{\dual}[1]{\left\langle {#1} \right\rangle}

\newcommand{\argmin}[0]{ {\mathop{{\rm  argmin}}\,}}

\newcommand{\nm}[1]{\left\lVert {#1} \right\rVert}
\newcommand{\snm}[1]{\left\lvert {#1} \right\rvert}

\newcommand{\ssnm}[1]
{
	\left\vert\kern-0.25ex
	\left\vert\kern-0.25ex
	\left\vert
	{#1}
	\right\vert\kern-0.25ex
	\right\vert\kern-0.25ex
	\right\vert
}

\makeatletter
\def\spher@harm#1{%
	\vbox{\hbox{%
			\offinterlineskip
			\valign{&\hb@xt@2\p@{\hss$##$\hss}\vskip.2ex\cr#1\crcr}%
		}\vskip-.36ex}%
}
\def\gshone{\spher@harm{.}}
\def\gshtwo{\spher@harm{.&.}}
\def\gshthree{\spher@harm{.&.&.}}
\let\gsh\spher@harm
\makeatother

\newtheorem{assum}{Assumption}

\newtheorem{lem}{Lemma}[section]

\newtheorem{rem}{Remark}[section]

\newtheorem{thm}{Theorem}[section]

\newcolumntype{I}{!{\vrule width 1,5pt}}
\newlength\savedwidth

\newlength\savewidth

\newcounter{mnote}
\let\oldmarginpar\marginpar
\renewcommand\marginpar[1]
{\-\oldmarginpar[\raggedleft\footnotesize #1]{\raggedright\footnotesize #1}}

\makeatletter\def\@captype{table}\makeatother

\newfloatcommand{capbtabbox}{table}[][\FBwidth]
\usepackage[T1]{fontenc} 
\usepackage[utf8]{inputenc} 
\usepackage{authblk}
\usepackage[all]{xy}
\begin{document}
	%	\title{
		%		\Large \bf An accelerated gradient flow approach to convex multiobjective programming}
	\title{\Large \bf An accelerated gradient method with adaptive restart for convex multiobjective optimization problems\thanks{This work was supported by the National Natural Science Foundation of China (Grant Nos. 12401402, 12431010, 11991024, 12171060), the Science and Technology Research Program of Chongqing Municipal Education Commission (Grant No. KJZD-K202300505), the Natural Science Foundation of Chongqing (Grant No. CSTB2024NSCQ-MSX0329) and the Foundation of Chongqing Normal University (Grant No. 202210000161).}}
	
	\author[,1,2]{Hao Luo\thanks{Email: luohao@cqnu.edu.cn}}
	
	\author[,1]{Liping Tang\thanks{Email: tanglipings@163.com}}
	
	\author[,1]{Xinmin Yang\thanks{Email: xmyang@cqnu.edu.cn}}		
	
	\affil[1]{National Center for Applied Mathematics in Chongqing, Chongqing Normal University, Chongqing, 401331, China}
	\affil[2]{Chongqing Research Institute of Big Data, Peking University,  Chongqing, 401121, China}

	\date{ }
	\maketitle
	
	\begin{abstract}
	In this work, based on the continuous time approach, we propose an accelerated gradient method with adaptive residual restart for convex multiobjective optimization problems. For the first, we derive rigorously the continuous limit of the multiobjective accelerated proximal gradient method by Tanabe et al. [Comput. Optim. Appl., 2023]. It is a second-order ordinary differential equation (ODE) that involves a special projection operator and can be viewed as an extension of the ODE by Su et al. [J. Mach. Learn. Res., 2016] for Nesterov acceleration. Then, we introduce a novel accelerated multiobjective gradient (AMG) flow with tailored time scaling that adapts automatically to the convex case and the strongly convex case, and the exponential decay rate of a merit function along with the solution trajectory of AMG flow is established via the Lyapunov analysis. After that, we consider an implicit-explicit time discretization and obtain an accelerated multiobjective gradient method with a convex quadratic programming subproblem. The fast sublinear rate and linear rate are proved respectively for convex and strongly convex problems. In addition, we present an efficient residual based adaptive restart technique to overcome the oscillation issue and improve the convergence significantly. Numerical results are provided to validate the practical performance  of the proposed method.
	\end{abstract}
	%	\medskip\noindent{\bf Keywords:} 
	%	\input{../../key}
	%	\begin{keywords}
		%xxx
		%\end{keywords}

		%		\tableofcontents
		
		%		\input{../intro}
		\section{Introduction}
		Multiobjective programming (MOP) arises in many practical applications, such as finance, economics and management science. Usually, it aims to minimize a vector valued function
		\begin{equation}\label{eq:min-mobj}
			\tag{MOP}
			\min_{x\in\R^n}\,F(x) =\left[f_1(x),\cdots,f_m(x)\right]^\top,
		\end{equation}
		where for $1\leq j\leq m,\,f_j:\R^n\to\R\cup\{+\infty\}$ is proper, closed and convex. Unlike the single objective case, \cref{eq:min-mobj} involves more than one objectives, but in most cases, one cannot minimize all objective functions simultaneously because there might be conflict among some of them. Therefore \cref{eq:min-mobj} has different concept of optimality, and this makes it more challenge than standard optimization problems for both theoretical analysis and numerical method. 
		
		Classical methods for solving \cref{eq:min-mobj} are mainly based on scalarization and heuristic approach. However, the former involves the fine-tuning on the weight parameter and it is not easy to approximate the Pareto front well. The latter is more likely to find the global Pareto front but it is rare to see theoretical convergence guarantee. Besides, neither of these two methods promise solution sequences with decreasing property, which is attractive for some applications \cite{Attouch2015c}. 
		\subsection{Multiobjective gradient methods}
		
		For single objective problems, the gradient descent moves along with the negative gradient, which stands for the steepest descent direction. Similar idea has been developed for multiobjective problems. In this case, $d$ is called a multiobjective descent direction at $x$ if $\dual{\nabla f_j(x),d}<0$ for all $1\leq j\leq m$.	In other words, $d$ is a common descent direction for all objectives. 
		The multiobjective  steepest descent direction at $x$ is defined by \cite{fliege_steepest_2000}
		\begin{equation}\label{eq:dx}
			d(x) =\argmin_{d\in\R^n}\left\{	\frac{1}{2}\nm{d}^2+	\max_{1\leq j\leq m}\dual{\nabla f_j(x),d}\right\},
		\end{equation}
		which can be understood equivalently as \cite{Attouch2015c}
		\[	d(x)/\nm{d(x)}=	 \argmin_{d\in\R^n}\left\{		\max_{1\leq j\leq m}\dual{\nabla f_j(x),d}:\,\nm{d}=1\right\}.
		\]  
		
	Based on the designing of descent directions, some multiobjective gradient methods are proposed. The steepest descent method \cite{fliege_steepest_2000,GranaDrummond2005} adopts the steepest descent direction \cref{eq:dx}, and there are extensions to the constrained case (projected gradient method) \cite{Drummond2004} and the nonsmooth case (subgradient and proximal gradient methods) \cite{Bonnel2005,DaCruzNeto2013,Miettinen1995,tanabeProximalGradientMethods2019}. Due to the slow convergence observed in multiobjective steepest descent methods, Chen et al. \cite{Chen2023e} incorporated the Barzilai-Borwein rule into the direction-finding subproblem and improved the theoretical analysis and practical performance. Concerning with a closed and bounded convex constraint under general partial orders, Chen et al. \cite{Chen2023g} developed a conditional gradient method based on the oriented distance function and established the convergence result. A generalization to the unbounded case can be found in Chen et al. \cite{chen_convergence_2024}.
	%{\color{red}Due to the slow convergence often observed in multiobjective steepest descent methods, Chen et al. \cite{Chen2023e} incorporated the Barzilai-Borwein (BB) rule into the direction-finding subproblem. By dynamically adjusting the gradient magnitudes, they generated a series of more efficient descent directions. This approach enhanced step sizes, as demonstrated by their theoretical analysis, and achieved rapid convergence, as confirmed through numerical experiments.
		%	For multiobjective optimization problems with closed convex constraints under general partial orders, Chen et al. \cite{Chen2023g} developed a conditional gradient algorithm based on the oriented distance function. When combined with an Armijo line search, the algorithm was proven to converge to the problem's stationary points. Then, new adaptive step size criteria and nonmonotone line search strategies were proposed, leading to convergence results under these two step size update rules.  
		%	In addition, Chen et al. \cite{chen_convergence_2024} introduced a conditional gradient algorithm for multiobjective optimization problems with unbounded feasible regions. A mild assumption on the gradient of the objective function was made by using the recession cone, ensuring the existence of optimal solutions to the subproblem for descent directions, which guaranteed the algorithm's well-definedness. Under reasonable assumptions, they proved the algorithm's convergence and established its sublinear convergence rate.  }
	By using the tool of merit function \cite{Tanabe2024}, Tanabe et al. \cite{tanabe_convergence_2023} established the convergence rate $\min\{L/k,(1-\mu/L)^k\}$ for the multiobjective proximal gradient method \cite{tanabeProximalGradientMethods2019}, where $\mu = \min_{1\leq j\leq m}\mu_j$ and $L=\max_{1\leq j\leq m}L_j$ (cf.\cref{assum:Lj-muj}). However, as mentioned in \cite{Chen2023e}, both the theoretical linear rate and the practical performance can be dramatically worse even if each objective is well conditioned (i.e., $L_j/\mu_j\approx1$).  To conquer this issue, Chen et al. \cite{chen_barzilai-borwein_2023-1} proposed two types of Barzilai-Borwein proximal gradient methods and obtained the improved linear rate $(1-\kappa)^k$ with $\kappa:=\min_{1\leq j\leq m}\mu_j/L_j$. For more related extensions, we refer to \cite{Assuncao2021,chen_convergence_2022,chen_barzilai-borwein_2023,chen_descent_2024,chen_subspace_24,chen_conjugate_2024} and the references therein.
		
		More recently, the acceleration technique has been applied to multiobjective gradient methods. Moudden and Mouatasim \cite{ElMoudden2020} presented an accelerated diagonal steepest descent method, and under the assumption that the dual variable with respect to the descent direction leaves unchanged, they proved the fast sublinear rate $L/k^2$ for smooth convex objectives. Tanabe et al. \cite{Tanabe2023a} proposed a multiobjective accelerated proximal gradient (APG) method for general composite objectives and established the same rate of convergence without such technical assumption in \cite{ElMoudden2020}. With careful chosen extrapolation parameter, Zhang and Yang \cite{zhang_convergence_2023} showed that APG converges with an improved rate $o(1/k^2)$. However, it seems unclear how to extend existing works to strongly convex problems with faster linear rate $(1-\sqrt{\mu/L})^k$.

		There are also multiobjective Newton methods \cite{chen_convergence_2023,Fliege2009} and Quasi-Newton methods \cite{Ansary2015,Povalej2014} that use second-order (Hessian) information to construct the descent directions. Theoretically we have (locally) superlinear or even quadratic convergence, but, the subproblem is not easy to solve and this will affect the overall performance. Instead, the above mentioned gradient-based methods involve a projection subproblem (cf.\cref{eq:dx}) and can be transferred to a convex Quadratic Programming (QP) in terms of the dual variable, which is usually low dimension and can be solved very efficiently.
		
		\subsection{Multiobjective gradient flow}
		The gradient flow approach presents a continuous time perspective of first-order gradient-based methods, and has attracted many attentions for single objective optimization problems \cite{Attouch2018,Attouch2019c,chen_first_2019,chen_unified_2021,Luo2023,Luo2024ax,luo_differential_2021,luo_unified_2025,Polyak1964,Shi2022,su_dierential_2016}. For \cref{eq:min-mobj}, the idea of dynamical system can be dated back to  \cite{Cornet1983,Henry1973,Smale1973}. It looks for some continuous trajectory that flows along with the negative descent direction. In view of \cref{eq:dx}, the steepest descent direction $d(x)$ is nothing but the negative of the projection of the zero vector onto the convex hull $C(x) = \conv{\nabla f_j(x):1\leq j\leq m}$. For convenience, we write $d(x) = -\proj_{C(x)}(0) $. This leads to the multiobjective gradient flow
		\begin{equation}\label{eq:MG}
			X'+ \proj_{C(X)}(0) = 0,
		\end{equation}
		which has been studied by Miglierina \cite{Miglierina2004}, Garrigos \cite{Garrigos2015} and Attouch et al. \cite{Attouch2015c,Attouch2014a}. The ODE \cref{eq:MG} actually corresponds to the continuous time limit of many multiobjective gradient methods, such as the steepest descent \cite{fliege_steepest_2000,GranaDrummond2005} and the proximal gradient  \cite{Bonnel2005,DaCruzNeto2013,Miettinen1995,tanabeProximalGradientMethods2019}. For solution existence and decreasing property of \cref{eq:MG} under nonsmooth setting, we refer to \cite{Attouch2015c}.
		
		In \cite{attouch_multiibjective_2015}, Attouch and Garrigos proposed an inertial multiobjective gradient (IMOG) system
		\begin{equation}\label{eq:imog}
			X''+\beta X'+\proj_{C(X)}(0)=0,\quad\beta>0,
		\end{equation}
		which is a direct generalization of the heavy-ball model \cite{Polyak1964}. However, to prove the decreasing property and convergence result, they assumed that the damping coefficient is bounded below $\beta >\sqrt{L}$. Unfortunately, this excludes the asymptotically vanishing damping $\beta =3/t$ \cite{Attouch2018,Attouch2019c,su_dierential_2016}. Recently, Sonntag and Peitz \cite{Sonntag2024a,Sonntag2024} introduced the multiobjective inertial gradient-like dynamical system with asymptotic vanishing damping (MAVD):
		\begin{equation}\label{eq:mimog}
			X''+\beta X''+\proj_{C(X)}(-X'')=0,
		\end{equation}
		where $\beta>0$ or $\beta = \alpha/t$ with $\alpha>0$. Compared with \cref{eq:imog}, they used the projection of $-X''$ onto the convex hull $C(x)$, which is a perturbation of the steepest descent direction $d(x)$ since $X''$ approaches to zero as the trajectory converges to the equilibrium state. More interestingly, the MAVD model \cref{eq:mimog} has an equivalent gradient-like presentation
		\[
		\beta X'+\proj_{C(X)+{X''}}(0)=0.
		\]  
		Sonntag and Peitz \cite{Sonntag2024a,Sonntag2024} established the decreasing property and convergence result, and also proved the convergence rate $1/t^2$ via the Lyapunov analysis. Based on proper numerical discretization of \cref{eq:mimog}, an accelerated multiobjective gradient method with the sublinear rate $L/k^2$ has been proposed in \cite{Sonntag2024}.
		
Here we mention that, existing models (eg.\cref{eq:imog,eq:mimog}) consider only convex problems. This motivates us to developing new continuous models and discrete methods with fast (linear) convergence for strongly convex case. Besides, the understanding of the MAVD system \cref{eq:mimog} deserves further investigation. As discussed in \cite[Section 6.5]{Sonntag2024}, the numerical discretization of MAVD has tight connection with APG \cite{Tanabe2023a}. But there is still lack of a rigorous proof for the continuous time limit of APG.
		\subsection{Main contributions}
		The main contributions of this work are summarized as below.
		\begin{itemize}
			\item Firstly, we derive the continuous time limit of APG \cite{Tanabe2023a}, which is the following ODE
			\begin{equation*} 
				\left\{
				\begin{aligned} 
					{}&\theta(t)X'(t) = Z(t)-X(t),\\
					{}& Z'(t) = -\theta(t)\proj_{C(X(t))}\left(\frac{3X'(t)-2Z'(t)}{2\theta(t)}\right),
				\end{aligned}
				\right.
			\end{equation*}
			where $\theta(t)=t/2+\theta_0$ with $\theta_0>0$. This is also equivalent to 
			\begin{equation*} 
				X''(t) + \frac{3}{t+2\theta_0}X'(t) + \proj_{C(X(t))}(-X''(t)) = 0.
			\end{equation*}
			We see that the above system is a special case of MAVD \cref{eq:mimog}. This agrees with the discussion in \cite[Section 6.5]{Sonntag2024} for the discrete case.
			\item After that, we propose a novel accelerated multiobjective gradient (AMG) flow with tailored time rescaling parameter
			\begin{equation}\label{eq:amg-intro}
				\gamma	X''+(\mu+\gamma)X'+\proj_{C(X)}(-\gamma X'') = 0,
			\end{equation}
			where  $\gamma$ satisfies $\gamma'=\mu-\gamma$ with $\gamma(0)=\gamma_0>0$. Recall that $\mu=  \min_{1\leq j\leq m}\mu_j\geq0$ is the smallest strong convexity constants (cf.\cref{assum:Lj-muj}). By using the Lyapunov analysis, we shall establish the exponential convergence rate $e^{-t}$ of a merit function, which measures the Pareto optimality with respect to \cref{eq:min-mobj}.
			\item Then we consider an implicit-explicit numerical scheme for our AMG flow \cref{eq:amg-intro} and obtain an accelerated multiobjective gradient method with the fast rate $\min\{L/k^2,\,(1-\sqrt{\mu/L})^k\}$ under \cref{assum:Lj-muj}. That is, for the convex case $\mu=0$, we have the same sublinear rate $L/k^2$ as that in \cite{Sonntag2024a,Sonntag2024,Tanabe2023a}, but for $\mu>0$, ours possesses linear convergence. Besides, our method involves a QP subproblem that appears in many existing methods \cite{fliege_steepest_2000,Sonntag2024a,Sonntag2024,tanabeProximalGradientMethods2019,Tanabe2023a}.
			\item Lastly, to overcome the oscillation phenomenon, we propose a novel residual based adaptive restart technique, which improves the practical performance of our method significantly. This also avoids using the global strong convexity parameter but maintain the fast linear convergence.
		\end{itemize}
		\subsection{Organization}
		The rest of this paper is organized as follows. In \cref{sec:pre} we prepare some essential preliminaries. Then in \cref{sec:dis2ode}, we shall derive the continuous time model the APG method \cite{Tanabe2023a}. After that, in \cref{sec:amg}, we propose our new model with tailored time scaling and present a Lyapunov analysis. An implicit-explicit numerical scheme with provable accelerated rate will be considered in \cref{sec:imex-qp}, and in \cref{sec:back-res} we present two variants with backtracking and adaptive restart to estimate the local Lipschitz constant and overcome the issue of oscillation. Some numerical results are reported in \cref{sec:num} and finally, we give some concluding remarks in \cref{sec:conclu}.

		\section{Preliminary}
		\label{sec:pre}

		\subsection{Pareto optimality}
		For $p,\,q\in\R^m$, we say $p<q$ if $p_j<q_j$ for all $1\leq j\leq m$. Likewise, the relation $p\leq q$ can be defined as well. A vector $y\in\R^n$ is {\it dominated} by $x\in\R^n$ with respect to \cref{eq:min-mobj} if $F(x)\leq F(y)$ and $ F(x)\neq F(y)$.
		Alternatively, we say $x$ dominates $y$ when $y$ is dominated by $x$. 
		
		A point $x^*\in\R^n$ is called {\it  weakly Pareto optimal} or a {\it weakly Pareto optimal solution}  or a {\it weak Pareto point} of \cref{eq:min-mobj} if there has no $y\in\R^n$ that $F(y)<F(x^*)$. The collection of all weakly Pareto optimal solutions is called the {\it weak Pareto set}, i.e.,
		\[
		\mathcal P_w := \{x^*\in \R^n:\, x^*\text{ is a weak Pareto point of \cref{eq:min-mobj} }\}.
		\]
		The image $F(\mathcal P_w)$ of the weak Pareto set $\mathcal P_w$ is called the {\it weak Pareto front}.
		
		A point $x^*\in\R^n$ is called {\it Pareto optimal} or a {\it Pareto optimal solution} or a {\it Pareto point} of \cref{eq:min-mobj} if there has no $y\in\R^n$ dominating $x^*$. The set of all Pareto optimal points is called the {\it Pareto set}:
		\[
		\mathcal P := \{x^*\in \R^n:\, x^*\text{ is a Pareto point of \cref{eq:min-mobj} }\}.
		\]
		The image $F(\mathcal P)$ of the Pareto set $\mathcal P$ is the {\it Pareto front}.
		\subsection{Functional class}
		Given any real-valued functional $f:\R^n\to\R$, its level set is defined by $\mathcal L_f(\alpha):=\{x\in\R^n:\,f(x)\leq \alpha\}$ for all $\alpha\in\R$. For a vector-valued function $F = \left[f_1,\cdots,f_m\right]^\top :\R^n\to\R^m$, denote by $	DF(x) =[\nabla f_1(x),\cdots,\nabla f_m(x)]$ the (transposed) Jacobian. When no confusion arises, we also define the level set of $F$ by $\mathcal L_F(p) = \{x\in\R^n:\,F(x)\leq p\}$ for any $p\in\R^n$.
		
		As usual, the bracket $\dual{\cdot,\cdot}$ denotes the inner product of vectors. Let $\mathcal F_{L}^1(\R^n)$ be the set of all $C^1$ functions with $L$-Lipschitzian gradients over $\R^n$ and $\mathcal S_{\mu}^1(\R^n)$ the set of all $C^1$ functions with convexity parameter $\mu\geq0$. The intersection of these two is $\mathcal S_{\mu,L}^{1,1}(\R^n)=\mathcal F_{L}^1(\R^n)\cap \mathcal S_{\mu}^1(\R^n)$. 
		\begin{lem}\label{lem:gd-lem}
			Assume $f\in\mathcal S_{\mu}^{1}(\R^n)$ and $(x,y,M)\in\R^n\times\R^n\times\R_+$ satisfies
			\begin{equation}\label{eq:cond-M}
				-	\dual{\nabla f(y),x-y}\leq f(y)-f(x)+\frac{M}{2}\nm{y-x}^2,
			\end{equation}
			then it holds that
			\[
			-	\dual{\nabla f(y),x-z}\leq f(z)-f(x)-\frac{\mu}{2}\nm{y-z}^2+\frac{M}{2}\nm{y-x}^2\quad\forall\,z\in\R^n.
			\]
		\end{lem}
		\begin{proof}
			Combining \cref{eq:cond-M} with $-	\dual{\nabla f(y),y-z}\leq f(z)-f(y)-\frac{\mu}{2}\nm{y-z}^2$ proves the lemma.
		\end{proof}
		\subsection{Projection operator}
		Let $\Delta_m :=\{\lambda=\left[\lambda_1,\cdots,\lambda_m\right]^\top\in\R_+^m:\,\sum_{j=1}^{m}\lambda_j=1\}$ be the unit simplex. Given a set of vectors $\{p_1,\cdots,p_m\} \subset\R^n$, its convex hull is defined by
		\[
		\conv{p_1,\cdots,p_m} := \left\{\sum_{j=1}^{m}\lambda_jp_j:\,\lambda=\left[\lambda_1,\cdots,\lambda_m\right]^\top\in\Delta_m\right\}.
		\]
		For any nonempty  closed convex subset $C\subset\R^n$ and $x\in\R^n$, the distance from $x$ to $C$ is $	d(x,C): = \min_{y\in C}\nm{y-x}$. The minimizer exists uniquely and is denoted as
		\[
		\proj_{C}(x): =\mathop{ \argmin}\limits_{y\in C}\frac{1}{2}\nm{y-x}^2.
		\]
		It is not hard to verify the translation decomposition
		\begin{equation}\label{eq:proj-transla}
			\proj_{x+C}(0) = x+ \proj_{C}(-x)\quad\forall\,x\in\R^n.
		\end{equation}
		Moreover, according to \cite[Example 4.1]{attouch_quantitative_1993}, $	\proj_{C}(\cdot)$ is H\"{o}lder continuous with respect to the subset $C$. Namely, given two closed convex subsets $C_1$ and $C_2$, we have 
		\begin{equation}\label{eq:holder-projC}
			\nm{	\proj_{C_1}(x)-	\proj_{C_2}(x)}^2\leq \rho \max\left\{\sup_{p\in C_1}d(p,C_2),\,\sup_{q\in C_2}d(q,C_1)\right\},
		\end{equation}
	 for any $x\in\R^n$ and $\rho\geq\nm{x}+d(x,C_1)+d(x,C_2)$. Since $d(x,C)\leq \nm{x}+d(0,C)$, we can also take $\rho\geq 3\nm{x}+d(0,C_1)+d(0,C_2)$.

		\subsection{First-order necessary condition}
		The optimality condition of \cref{eq:min-mobj} reads as
		\begin{equation}\label{eq:1st-opt-cond-conv}
			0\in \conv{\nabla f_1(x^*),\cdots,\nabla f_m(x^*)},
		\end{equation}
		which is also called the Karush--Kuhn--Tucker (KKT) condition. Namely, there exists some dual variable $\lambda^*\in\Delta_m$ such that $DF(x^*)\lambda^*=\sum_{j=1}^m\lambda^*_j\nabla f_j(x^*) = 0$.
		For simplicity, we set $C(x) := \conv{\nabla f_j(x):1\leq j\leq m}$ for all $x\in\R^n$. Then the optimality condition \cref{eq:1st-opt-cond-conv} is equivalent to
		\begin{equation}\label{eq:1st-opt-cond-proj}
			\proj_{C(x^*)}(0) = 0.
		\end{equation}
		If $x^*$ satisfies the KKT condition \cref{eq:1st-opt-cond-conv} or \cref{eq:1st-opt-cond-proj}, then we call it {\it Pareto critical}. The collection of all Pareto critical points is called the {\it Pareto critical set}:
		\[
		\mathcal P_c := \{x^*\in \R^n:\, x^*\text{ is a Pareto critical point of \cref{eq:min-mobj} }\}.
		\]
		
		By \cite[Lemma 1.3]{Attouch2015c}, we know that $\mathcal P_w\subset\mathcal P_c$ and by convexity the converse holds true, i.e., $\mathcal P_w=\mathcal P_c$ for convex problems. Hence, analogously to the single objective case ($m=1$), the Pareto criticality is a necessary condition for weak Pareto optimality, and in the convex setting, the Pareto criticality is also sufficient for optimality. Moreover, if each $f_j$ is strictly convex, then we have $\mathcal P=\mathcal P_c$. 
		\subsection{Merit function}
		A merit function associated with \cref{eq:min-mobj} is a nonnegative function that attains zero only at weakly Pareto optimal solutions. An overview on merit functions for multiobjective optimization is given in \cite{Tanabe2024}. In this paper, we consider the merit function
		\begin{equation}\label{eq:merit}
			u_0(x) := \sup_{z\in\R^n}\left\{f(x;z):=\min_{1\leq j\leq m} \left[f_j(x)-f_j(z)\right]\right\}.
		\end{equation}
		
		The following result has been prove in \cite[Theorem 3.1]{Tanabe2024}, which motivates the usage of $u_0(x)$ as a measure of complexity for multiobjective optimization methods. 
		\begin{thm}[\cite{Tanabe2024}]
			The function $u_0:\R^n\to\R\cup\{+\infty\}$ defined by \cref{eq:merit} is nonnegative and lower semicontinuous. Moreover, $x^*\in\mathcal P_w$ if and only if $u_0(x^*) = 0$.  
		\end{thm} 
		
		The merit function $u_0$ defined by \cref{eq:merit} involves the supermum of $f(x;z)$ over the whole space. This is not convenient for later convergence rate analysis. Fortunately, we can restrict it to the bounded level set, under proper assumptions.
		\begin{assum}\label{assum:Lj-muj}
			For $1\leq j\leq m,\,f_j\in\mathcal F_{\mu_j,L_j}^{1,1}(\R^n)$ with $0\leq \mu_j\leq L_j<+\infty$. For simplicity, set $\mu := \min_{1\leq j\leq m}\mu_j$ and $L:=\max_{1\leq j\leq m}L_j$.
		\end{assum}
		\begin{assum}\label{assum:j0}
			There exists $1\leq j_0\leq m$ such that the level set $\mathcal L_{f_{j_0}}(\alpha)=\{x\in\R^n:\,f_{j_0}(x)\leq \alpha\}$ is bounded for all $\alpha\in\R$. In other words, the quantity is finite
			\[
			R_{j_0}(\alpha): = \sup_{x\in \mathcal L_{f_{j_0}}(\alpha)}\nm{x}<+\infty
			\quad\forall\,\alpha\in\R.
			\]
		\end{assum}
		\begin{assum}\label{assum:alpha-pw}
			Let $\alpha\in\R^n$ be such that the level set $\mathcal L_F(\alpha)$ is nonempty. For all  $x\in\mathcal L_F(\alpha)$, there exists $x^*\in P_w\cap \mathcal L_F(F(x))$.
		\end{assum}
		\begin{lem}
			\label{lem:u0-}
			Let $\alpha\in\R^n$ be such that the level set $\mathcal L_F(\alpha)$ is nonempty. Then we have the following.
			\begin{itemize}
				\item[(i)] Under \cref{assum:j0}, we have 	 $
				S(\alpha):=\sup_{F^*\in F(P_w\cap \mathcal L_F(\alpha))}\inf_{z\in F^{-1}(F^*)}\nm{z}<+\infty
				$.
				\item[(ii)] Under \cref{assum:alpha-pw}, for all  $x\in\mathcal L_F(\alpha)$, we have 
				\[
				u_0(x)=
				\sup_{F^*\in F(P_w\cap \mathcal L_F(\alpha))}\inf_{z\in F^{-1}(F^*)}f(x;z).
				\]	
			\end{itemize}
		\end{lem}
		\begin{proof}
			Consider any $F^*\in F(P_w\cap \mathcal L_F(\alpha))$ with $z^*\in P_w\cap \mathcal L_F(\alpha)$ such that $F^*=F(z^*)$. It follows that $	\inf_{z\in F^{-1}(F^*)}\nm{z}\leq \nm{z^*}$.
			Since $z^*\in  \mathcal L_F(\alpha)$, by \cref{assum:j0}, we conclude that $z^*\in\mathcal L_{f_{j_0}}(\alpha)$, and thus 
			\[
			\inf_{z\in F^{-1}(F^*)}\nm{z}	\leq \sup_{z\in\mathcal L_{f_{j_0}}(\alpha)}\nm{z}=R_{j_0}(\alpha)<+\infty.
			\]
			Hence, we obtain $
			S(\alpha)\leq  R_{j_0}( \alpha)<+\infty$.
			This proves the first claim $(i)$.
			
			Then, let us check the second one $(ii)$. The proof can be found in \cite[Theorem 5.2]{tanabe_convergence_2023}. For the sake of completeness, we provide more details. By the definition \cref{eq:merit}, it is clear that $u_0(x)\geq 
			\sup_{F^*\in F(P_w\cap \mathcal L_F(\alpha))}\inf_{z\in F^{-1}(F^*)}f(x;z)$. Let us establish the reverse estimate. It is not hard to obtain
			\[
			u_0(x) = \sup_{z\in \mathcal L_F(F(x))}f(x;z)= \sup_{z\in \mathcal L_F(\alpha)}f(x;z).
			\]
			By \cref{assum:alpha-pw}, for all $z\in\mathcal L_F(\alpha)$, there is $x^*\in P_w\cap\mathcal L_F(F(z))$, and thus $f(x;z)\leq f(x;x^*)$, which gives that 
			\[
			\begin{aligned}
				u_0(x)\leq {}&\sup_{x^*\in P_w\cap\mathcal L_F(F(z))}f(x;x^*)\leq \sup_{x^*\in P_w\cap\mathcal L_F(\alpha)}f(x;x^*)\\
				={}&\sup_{F^*\in F(P_w\cap \mathcal L_F(\alpha))}\min_{1\leq j\leq m}[f_j(x)-f_j^*]\\
				={}&\sup_{F^*\in F(P_w\cap \mathcal L_F(\alpha))}\inf_{z\in F^{-1}(F^*)}\min_{1\leq j\leq m}[f_j(x)-f_j(z)]\\
				={}&\sup_{F^*\in F(P_w\cap \mathcal L_F(\alpha))}\inf_{z\in F^{-1}(F^*)}f(x;z).		
			\end{aligned}
			\]
			Consequently, we get the identity $u_0(x)	=\sup_{F^*\in F(P_w\cap \mathcal L_F(\alpha))}\inf_{z\in F^{-1}(F^*)}f(x;z)$ and completes the proof of this lemma.
		\end{proof}

		\section{Deriving the Continuous Model}
		\label{sec:dis2ode}
		In this section, we shall drive the continuous time limit of the multiobjective accelerated proximal gradient (APG) method proposed by Tanabe et al. \cite[Algorthm 2]{Tanabe2023a}, which is a generalization of FISTA and works for the composite objective $f_j = h_j+g_j$, where $h_j$ is smooth and $g_j$ is nonsmooth. For simplicity, we focus on the smooth setting $f_j\in\mathcal F_{L_j}^1(\R^n)$ for $1\leq j\leq m$. In this case, APG reads as follows: given $x_0=y_0\in\R^n$ and $\theta_0>0$, compute
		\begin{equation}\label{eq:apg-mop}
			\tag{APG}
			\left\{
			\begin{aligned}
				{}&	\theta_{k+1}^{-1}=\sqrt{\theta_k^{-2}+1 / 4}+1 / 2,\\
				{}&	x_{k+1} =  p_{\tau}^{{acc}}\left(x_{k}, y_k\right)=\mathop{\argmin}\limits_{z\in\R^n} \varphi_{\tau}^{acc}(z; x_{k}, y_k),\\
				{}&		y_{k+1} =x_{k+1}+ \theta_{k+1}\left(\theta_k^{-1}-1\right) \left(x_{k+1}-x_{k}\right),
			\end{aligned}
			\right.	
		\end{equation}
		where $\tau\in(0,1/L]$ and 
		\[
		\varphi_{\tau}^{acc}(z; x_k,y_k) : = \max _{1\leq j\leq m}\!\left\{\left[\left\langle\nabla f_j(y_k), z-y_k\right\rangle+f_j(y_k)-f_j(x_k)\right]+\frac{1}{2\tau}\|z-y_k\|^2\right\}.
		\]
		
		According to \cite[Theorem 5.1]{Tanabe2023a}, under \cref{assum:j0}, both $\{x_k\}_{k=0}^\infty$ and $\{y_k\}_{k=0}^\infty$ are bounded:
		\begin{equation}\label{eq:bd-xk-yk}
			\nm{x_k}+\nm{y_k}\leq C_1<+\infty\quad\forall\,k\geq0.
		\end{equation}
		With this property, we have the following main result. 
		\begin{thm}\label{thm:ode-apg-mop}
			Assume $f_j\in\mathcal F_{L_j}^1(\R^n)$ for $1\leq j\leq m$ and \cref{assum:j0} holds true, then the continuous-time limit of \cref{eq:apg-mop} is the following ODE
			\begin{equation}\label{eq:ode-apg-mop}
				\left\{
				\begin{aligned}
					{}&[\theta^2(t)]'=\theta(t),\\
					{}&\theta(t)X'(t) = Z(t)-X(t),\\
					{}& Z'(t) = -\theta(t)\proj_{C(X(t))}\left(\frac{3X'(t)-2Z'(t)}{2\theta(t)}\right),
				\end{aligned}
				\right.
			\end{equation}
			with the initial conditions $\theta(0)=\theta_0>0,\,X(0) = x_0\in\R^n$ and $Z(0) = z_0\in\R^n$,
			which is equivalent to \begin{equation}\label{eq:2ndode-apg-mop}
				X''(t) + \frac{3}{t+2\theta_0}X'(t) + \proj_{C(X(t))}(-X''(t)) = 0.
			\end{equation}
		\end{thm} 
		As we can see, in the continuous level, the limit ODE of \cref{eq:apg-mop} is still a generalization of the asymptotically vanishing damping model for NAG and FISTA derived by Su et al. \cite{su_dierential_2016}. We also notice that the second-order model \cref{eq:2ndode-apg-mop} has been proposed by Sonntag and Peitz \cite{Sonntag2024a,Sonntag2024}. In the sequel, to prove \cref{thm:ode-apg-mop}, we provide some auxiliary estimates in \cref{sec:aux-pf-ode} and then complete the detailed proof in \cref{sec:pf-ode}.
		\subsection{Auxiliary estimates}
		\label{sec:aux-pf-ode}
		We change the minimax order to get the dual problem
		\[
		\begin{aligned}
			{}&	\mathop{\min}\limits_{z\in\R^n} \varphi_{\tau}^{acc}(z; x_k,y_k)\\
			=	{}&\mathop{\min}\limits_{z\in\R^n} \max _{1\leq j\leq m}\left\{\left[\left\langle\nabla f_j(y_k), z-y_k\right\rangle+f_j(y_k)-f_j(x_k)\right]+\frac1{2\tau}\|z-y_k\|^2\right\}\\
			={}&\max _{\lambda\in\Delta_m}\mathop{\min}\limits_{z\in\R^n} \left\{\dual{\lambda, [DF(y_k)]^\top (z-y_k) +F(y_k)-F(x_k)} +\frac{1}{2\tau }\|z-y_k\|^2\right\}.	
		\end{aligned}
		\]
		Given $\lambda\in\Delta_m$, the minimization problem admits a unique solution $z(\lambda)=y_k-\tau DF(y_k)\lambda$, and it follows that
		\[
		\begin{aligned}
			\mathop{\min}\limits_{z\in\R^n} \varphi_{\tau}^{acc}(z; x_k,y_k)	
			={}&\max _{\lambda\in\Delta_m}\left\{\dual{\lambda,F(y_k)-F(x_k)}-\frac\tau{2}\nm{DF(y_k)\lambda}^2\right\}.		
		\end{aligned}
		\]
		In particular, we have 
		\begin{equation}\label{eq:pacc}
			p_{\tau}^{{acc}}\left(x_k,y_k\right) = y_k-\tau DF(y_k)\lambda_{\tau}(x_k,y_k),
		\end{equation}
		where 
		\begin{equation}\label{eq:argmin-lambda}
			\lambda_{\tau}(x_k,y_k)\in{} \mathop{\argmin} _{\lambda\in\Delta_m}\left\{\dual{\lambda,F(x_k)-F(y_k)}+\frac\tau{2}\nm{DF(y_k)\lambda}^2\right\}. 
		\end{equation}
		
		If we ignore the linear term $\dual{\lambda,F(x_k)-F(y_k)}$ in \cref{eq:argmin-lambda}, then it is a standard quadratic projection problem. In what follows, we establish some auxiliary estimates, which tell us how far is this problem away from the standard projection.
		\begin{lem}\label{lem:est-D-D}
			Assume $f_j\in\mathcal F_{L_j}^1(\R^n)$ for $1\leq j\leq m$. We have 
			\[
			\nm{DF(y_k)\lambda_{\tau}(x_k,y_k) - \proj_{C(y_k)}\left(\frac{y_k-x_k}{\tau}\right)}\leq \sqrt{\frac{2L}{ \tau }}\nm{x_k-y_k}.
			\]
		\end{lem}
		\begin{proof}
			Define $\widehat{\lambda}_\tau(x_k,y_k)\in\Delta_m$ by that
			\[
			\begin{aligned}
				\widehat{\lambda}_\tau(x_k,y_k)\in{}&\mathop{\argmin} _{\lambda\in\Delta_m}\,\phi(\lambda):=\frac1{2}\nm{DF(y_k)\lambda-\frac{y_k-x_k}{\tau }}^2,
			\end{aligned}
			\]
			which implies $\proj_{C(y_k)}\left((y_k-x_k)/\tau\right) =DF(y_k)\widehat{\lambda}_\tau(x_k,y_k) $. In view of \cref{eq:argmin-lambda}, we have
			\[
			\begin{aligned}
				\lambda_{\tau}(x_k,y_k)\in{} &\mathop{\argmin} _{\lambda\in\Delta_m}\left\{  \phi(\lambda)+\frac{ 1}{\tau}\dual{\lambda,E_k} \right\}.
			\end{aligned}
			\]
			where $E_k:= F(x_k)-F(y_k)-[DF(y_k)]^\top (x_k-y_k) $ satisfies that $\nm{E_k}_{\infty}\leq L/2\nm{x_k-y_k}^2$. Using the optimality conditions of $\lambda_{\tau}(x_k,y_k)$ and $\widehat{\lambda}_\tau(x_k,y_k)$ implies that
			\[
			\begin{aligned}
				\phi(\lambda_{\tau}(x_k,y_k))-\phi(\lambda)\leq \frac{1}{\tau}\dual{\lambda-\lambda_{\tau}(x_k,y_k),E_k}	{}&,\quad\forall\,\lambda\in\Delta_m,\\
				\dual{\nabla \phi(\widehat{\lambda}_\tau(x_k,y_k)), \lambda-\widehat{\lambda}_\tau(x_k,y_k)}\geq	 0{}&,\quad\forall\,\lambda\in\Delta_m.
			\end{aligned}
			\]
			This yields the following estimate
			\[
			\begin{aligned}
				{}&\frac{1}{2}\nm{DF(y_k)\lambda_{\tau}(x_k,y_k) - \proj_{C(y_k)}\left(\frac{y_k-x_k}{\tau}\right)}^2\\
				= {}&	\frac{1}{2}\nm{DF(y_k)(\lambda_{\tau}(x_k,y_k)-\widehat{\lambda}_\tau(x_k,y_k))}^2 \\={}& \phi(\lambda_{\tau}(x_k,y_k)) -  \phi(\widehat{\lambda}_\tau(x_k,y_k)) -\dual{\nabla \phi(\widehat{\lambda}_\tau(x_k,y_k)),{\lambda}_\tau(x_k,y_k)-\widehat{\lambda}_\tau(x_k,y_k)}\\
				\leq{}&\frac{1}\tau\dual{E_k,\widehat{\lambda}_\tau(x_k,y_k)-{\lambda}_\tau(x_k,y_k)}\leq \frac{1}\tau\nm{E_k}_{\infty}\nm{\widehat{\lambda}_\tau(x_k,y_k)-{\lambda}_\tau(x_k,y_k)}_{1}\\
				\leq{}&\frac{2}\tau\nm{E_k}_{\infty}\leq \frac{L}{\tau}\nm{x_k-y_k}^2.
			\end{aligned}
			\]
			This completes the proof of this lemma.
		\end{proof}
		\begin{lem}\label{lem:est-proj-proj}
			Assume $f_j\in\mathcal F_{L_j}^1(\R^n)$ for $1\leq j\leq m$ and \cref{assum:j0} holds true, then
			\begin{equation}\label{eq:est-3.2}
				\begin{aligned}
					{}&	\nm{\proj_{C(y_k)}\left(\frac{y_k-x_k}{\tau}\right)-\proj_{C(x_k)}\left(\frac{y_k-x_k}{\tau}\right)}\\
					\leq{}&   \sqrt{C_2L}\nm{x_k-y_k}^{1/2}+\sqrt{\frac{3L}{\tau}}\nm{x_k-y_k},
				\end{aligned}
			\end{equation}
			where $C_2:=LC_1+2\max_{1\leq j\leq m}\nm{\nabla f_j(0)} $ and $C_1>0$ is the constant in \cref{eq:bd-xk-yk}.
		\end{lem}
		\begin{proof}
			Thanks to \cref{eq:holder-projC}, we have 
			\[
			\begin{aligned}
				{}&		\nm{\proj_{C(y_k)}\left(\frac{y_k-x_k}{\tau}\right)-\proj_{C(x_k)}\left(\frac{y_k-x_k}{\tau}\right)}^2\leq \rho \Lambda_\tau(x_k,y_k),
			\end{aligned}
			\]
			where $\rho= 3\nm{(y_k-x_k)/\tau}+d\left(0,C(x_k)\right)+d\left(0,C(y_k)\right)$ and 
			\[
			\Lambda_\tau(x_k,y_k):=\max\left\{\sup_{p\in C(x_k)}d(p,C(y_k)),\,\sup_{q\in C(y_k)}d(q,C(x_k))\right\}.
			\]
			An elementary computation gives
			\[
			\begin{aligned}
				\rho\leq {}&3\nm{\frac{y_k-x_k}{\tau}}+\max_{1\leq j\leq m}\nm{\nabla f_j(x_k)}+\max_{1\leq j\leq m}\nm{\nabla f_j(y_k)}\\
				\leq {}&3\nm{\frac{y_k-x_k}{\tau}}+2\max_{1\leq j\leq m}\nm{\nabla f_j(0)}+\max_{1\leq j\leq m}\nm{\nabla f_j(x_k)-\nabla f_j(0)}\\
				{}&\quad +\max_{1\leq j\leq m}\nm{\nabla f_j(y_k)-\nabla f_j(0)}\\
				\leq {}&3\nm{\frac{y_k-x_k}{\tau}}+2\max_{1\leq j\leq m}\nm{\nabla f_j(0)}+L(\nm{x_k}+\nm{y_k})\\
				\leq {}&3\nm{\frac{y_k-x_k}{\tau}}+2\max_{1\leq j\leq m}\nm{\nabla f_j(0)}+LC_1,	
			\end{aligned}
			\]
			in the last step we used \cref{eq:bd-xk-yk}.
			Given any $p = \sum_{j=1}^{m}\lambda_j\nabla f_j(x_k)\in C(x_k)$ with $\lambda\in\Delta_m$, we have
			\[
			\begin{aligned}
				d(p,C(y_k))\leq{}& \nm{\sum_{j=1}^{m}\lambda_j(\nabla f_j(x_k)-\nabla f_j(y_k))}\leq \sum_{j=1}^{m}\lambda_j\nm{\nabla f_j(x_k)-\nabla f_j(y_k)}\\
				\leq{}& \sum_{j=1}^{m}\lambda_jL_j\nm{x_k-y_k}\leq L\nm{x_k-y_k},
			\end{aligned}
			\]
			and similarly it holds that $d(q,C(x_k))\leq L\nm{x_k-y_k}$ for any $q\in C(y_k)$. Consequently, we obtain $\Lambda_\tau(x_k,y_k)\leq L\nm{x_k-y_k}$ and this leads to the desired estimate \cref{eq:est-3.2}.
		\end{proof}
		
		\begin{lem}\label{lem:y-p-proj}
			Assume $f_j\in\mathcal F_{L_j}^1(\R^n)$ for $1\leq j\leq m$ and \cref{assum:j0} holds true,	we have 
			\begin{equation}\label{eq:est-D-proj}\small
				\begin{aligned}
					{}&\nm{DF(y_k)\lambda_{\tau}(x_k,y_k) - \proj_{C(x_k)}\left(\frac{y_k-x_k}{\tau }\right)} \\\leq{}& \sqrt{C_2L}\nm{x_k-y_k}^{1/2}+\sqrt{\frac{10L}{\tau}}\nm{x_k-y_k},
				\end{aligned}
			\end{equation}
			where  $C_2>0$ is define in \cref{eq:est-3.2}.
		\end{lem}
		\begin{proof}
			Combining \cref{lem:est-D-D,lem:est-proj-proj} and the triangle inequality proves \cref{eq:est-D-proj}.
		\end{proof}
		\subsection{Proof of \cref{thm:ode-apg-mop}}
		\label{sec:pf-ode}
		Thanks to \cref{eq:pacc,lem:y-p-proj}, we have 
		\[
		\begin{aligned}
			x_{k+1}
			={}&	y_k-\tau DF(y_k)\lambda_\tau(x_k,y_k)=y_k-\tau \proj_{C(x_{k})}\left(\frac{y_k-x_k}{\tau }\right)-\tau \Delta_k,
		\end{aligned}
		\]
		where $\Delta_k=DF(y_k)\lambda_\tau(x_k,y_k)-\proj_{C(x_{k})}\left(\frac{y_k-x_k}{\tau }\right)$ and 
		\begin{equation}\label{eq:est-Dk}
			\nm{\Delta_k}\leq \sqrt{C_2L}\nm{x_k-y_k}^{1/2}+\sqrt{\frac{10L}{\tau}}\nm{x_k-y_k}.
		\end{equation} 
		
		Define an auxiliary sequence $\{z_k\}_{k=0}^\infty$ by that
		\[
		z_{k+1} = x_{k+1}+	\left(\theta_k^{-1}-1\right)\left(x_{k+1}-x_k\right),
		\]
		which gives $x_{k+1}-y_k
		={} \theta_k(z_{k+1}-z_k)$.
		Following the idea from \cite{Luo2024av}, we rewrite \cref{eq:apg-mop} as a finite difference template
		%%%%%%%%%%%%%%middle form%%%%%%
		%\begin{equation}\label{eq:diff-mNAG}
		%	\left\{
		%	\begin{aligned}
			%		\theta_k(z_{k+1}-z_k)= 			&-\tau \proj_{C(x_{k})}\left(\frac{y_k-x_k}{\tau }\right)-\tau E_k,\\
			%		x_{k+1} -x_k=			{}&\theta_k(	z_{k+1} -x_k),\\
			%		\theta_{k+1}^{-2}-\theta_k^{-2}={}&\theta_{k+1}^{-1}.
			%	\end{aligned}
		%	\right.
		%\end{equation}
		%%%%%%%%%%%%%%middle form%%%%%%
		\begin{subnumcases}{\label{eq:diff-mop-agd}}
			\frac{\left[\frac{\sqrt{\tau}}{\theta_{k+1}}\right]^{2}-\left[\frac{\sqrt{\tau}}{\theta_k }\right]^{2}}{\sqrt{\tau}}=	{}\frac{\sqrt{\tau}}{\theta_{k+1}},\\
			\frac{\sqrt{\tau}}{\theta_k }\frac{x_{k+1}-x_k}{\sqrt{\tau}}=	{} z_{k+1}  -x_k,\\
			\frac{z_{k+1}-z_k}{\sqrt{\tau}}=	{} - \frac{\sqrt{\tau}}{\theta_k }  \left[\proj_{C(x_{k})}\left(\frac{y_k-x_k}{\tau }\right)+\Delta_k\right].
			\label{eq:zk1-zk}
		\end{subnumcases}
		According to \cref{thm:sub-prob-case1}, \eqref{eq:zk1-zk} is equivalent to 
		\[
		\frac{z_{k+1}-z_k}{\sqrt{\tau}}= - \frac{\sqrt{\tau}}{\theta_k }  \left[\proj_{C(x_{k})}\left(\frac{\theta_k }{\sqrt{\tau}}\left[\frac{3}{2}\frac{x_{k+1}-x_k}{\sqrt{\tau}}-\frac{z_{k+1}-z_k}{\sqrt{\tau}}\right]+\frac{3\theta_k-2}{3\theta_k}\tau \Delta_k\right)+\Delta_k\right].
		\] 
		By \cite[Lemma 4.2]{Tanabe2023a}, it is not hard to find that $\theta_k\geq \theta_0+k/2$ for all $k\geq 0$, which yields the following two facts:
		\begin{itemize}
			\item $1-2/(3\theta_0)\leq (3\theta_k-2)/(3\theta_k)\leq 1$,
			\item $\nm{\proj_{C(x_{k})}\left(\frac{\theta_k }{\sqrt{\tau}}\left[\frac{3}{2}\frac{x_{k+1}-x_k}{\sqrt{\tau}}-\frac{z_{k+1}-z_k}{\sqrt{\tau}}\right]\right)-\proj_{C(x_{k})}\left(\frac{y_k-x_k}{\tau }\right)}\leq C(\theta_0)\tau\nm{\Delta_k}$,
		\end{itemize}
		where $C(\theta_0):=\max\{1,\snm{1-2/(3\theta_0)}\}$.
		
		We then introduce the time $t_k = \sqrt{\tau}k$ and the ansatzs
		\[
		X(t_k) = x_k,\quad Z(t_k)=z_k,\quad \theta(t_k) = \frac{\sqrt{\tau}}{\theta_k },
		\]
		with some smooth trajectories $(X(t),\,Z(t))$ and parameter $\theta(t)$ for $t\geq 0$. 
		We obtain from \cref{eq:diff-mop-agd} that
		\begin{equation}\label{eq:mid-}
			\left\{
			\begin{aligned}
				{}&	\frac{\theta^2(t_{k+1})-\theta^2(t_k)}{\sqrt{\tau}}=	{}\theta(t_{k+1}),\\
				{}&	\theta(t_k)\frac{X(t_{k+1})-X(t_k)}{\sqrt{\tau}}=	{} Z(t_{k+1})  -X(t_k),\\
				{}&	\frac{Z(t_{k+1})-Z(t_k)}{\sqrt{\tau}}=	{} - \theta(t_k)  \proj_{C(X(t_k))}\left(W\right)+O(\sqrt{\tau})\nm{\Delta_k}, 
			\end{aligned}
			\right.
		\end{equation}
		where 
		\[
		W:=\frac{3}{2}\frac{X(t_{k+1})-X(t_k)}{\theta(t_k)\sqrt{\tau}}-\frac{Z(t_{k+1})-Z(t_k)}{\theta(t_k)\sqrt{\tau}}.
		\]
		Note that for fixed $k\in\mathbb N$ we have 
		\[
		y_k-x_k = \theta_k(z_k-x_k) = \frac{\theta_k}{\sqrt{\tau}}\cdot \sqrt{\tau}(z_k-x_k)  = \sqrt{\tau}\cdot \frac{Z(t_k)-X(t_k)}{\theta(t_k)} =O(\sqrt{\tau}),
		\]
		which together with \cref{eq:est-Dk} implies $\nm{\Delta_k}=O(1) $. Thus, taking the limit $\tau\to0$ in \cref{eq:mid-} gives
		\[
		\left\{
		\begin{aligned}
			{}&[\theta^2(t)]' = \theta(t),\\
			{}&\theta(t)X'(t) = Z(t)-X(t),\\
			{}& Z'(t) = -\theta(t)\proj_{C(X(t))}\left(\frac{3X'(t)-2Z'(t)}{2\theta(t)}\right).
		\end{aligned}
		\right.
		\]
		This finishes the proof of \cref{thm:ode-apg-mop}.

		\section{Accelerated Multiobjective Gradient Flow}
		\label{sec:amg}
		In the rest of this paper, suppose \cref{assum:Lj-muj,assum:j0,assum:alpha-pw} hold true. The continuous model \cref{eq:2ndode-apg-mop} is used only for convex objectives. We borrow the time scaling idea from \cite{luo_differential_2021} to handle the convex case and strongly convex case in a unified way. 
		
		To do so, we introduce the scaling parameter $\gamma:[0,\infty)\to\R_+$ by that
		\begin{equation}\label{eq:gama}
			\gamma' = \mu-\gamma,
		\end{equation}
		with arbitrary positive initial condition $\gamma(0)=\gamma_0>0$. It is clear that $\gamma(t)=\mu+(\gamma_0-\mu)e^{-t}$ is positive for all $t>0$. Our new continuous model, called Accelerated Multiobjective Gradient (AMG) flow, reads as follows
		\begin{equation}\label{eq:AMG}
			\tag{AMG}
			\gamma X'' + (\mu+\gamma)X'+\proj_{C(X)}(-\gamma X'') = 0,
		\end{equation}
		with the initial conditions $X(0)=x_0\in\R^n$ and $X'(0)=x_1\in\R^n$. By the translation decomposition \cref{eq:proj-transla}, we also have a gradient-like formulation
		\begin{equation}\label{eq:AMG'} (\mu+\gamma)X'+\proj_{\gamma X''+C(X)}(0) = 0.
		\end{equation}
		\subsection{Equivalent models}
		Let $Z: = X + X' $ and rewrite \cref{eq:AMG} together with the parameter equation \cref{eq:gama} as a first-order system
		\begin{equation}\label{eq:AMG-sys} 
			\tag{AMG-sys}
			\left\{
			\begin{aligned}
				{}&\gamma' = \mu-\gamma,\\
				{}&		X' = Z-X,\\
				{}&		\gamma Z'  = \mu(X -Z ) -	\proj_{C(X )}(\gamma X'-\gamma Z' ),
			\end{aligned}
			\right.
		\end{equation}
		with the initial conditions $X(0)=x_0$ and $Z(0)=x_0+x_1$. Note that this is an implicit model since the right hand side depends on $X'$ and $Z'$. Applying \cref{thm:sub-prob-case1} to the equation for $Z'$ gives
		%%%%%%%%%%
		%\[
		%x^*=u-\proj_{C(x)}(w-x^*)\quad\Longrightarrow\quad x^* = u-v^*,
		%\]
		%%%%%%%%%%
		\[
		\begin{aligned}
			\gamma Z' \in {}& \mu(X -Z )-{\argmin}_{v\in C(X )}\dual{\mu(X-Z)-\gamma X' ,v}\\
			= {}& \mu(X -Z )-{\argmin}_{v\in C(X )}\dual{(\mu+\gamma)(X-Z) ,v}\\	
			={}& \mu(X -Z )- {\argmin}_{v\in C(X )}\dual{X-Z ,v},
		\end{aligned}
		\]
		in view of the facts that $\mu+\gamma>0$ and $X'=Z-X$. This leads to an equivalent autonomous differential inclusion
		\begin{equation}\label{eq:AMG-sys-ex}   
			\tag{AMG-DI}
			\left\{
			\begin{aligned}
				{}&\gamma' = \mu-\gamma,\\
				{}&		X' = Z-X,\\
				{}&		\gamma Z'  \in \mu(X -Z ) - {\argmin}_{v\in C(X )}\dual{X-Z ,v}.
			\end{aligned}
			\right.
		\end{equation}
		
		In fact,  we can consider a more 	generalized variant of \cref{eq:AMG-sys}:
		\begin{equation}\label{eq:AMG-sys'}
			\tag{AMG-sys'}
			\left\{
			\begin{aligned}
				{}&	\gamma' = \mu-\gamma.\\
				{}&		X' = Z-X,\\
				{}&		\gamma Z'  = \mu(X -Z ) -	\proj_{C(X )}(W),
			\end{aligned}
			\right.
		\end{equation} 
		with $	W = \alpha (Z -X )+\beta X' -\gamma Z' $, where $\alpha$ and $\beta$ are arbitrary such that $\alpha+\beta+\mu>0$. Again, using \cref{thm:sub-prob-case1} gives 
		\[
		\proj_{C(X )}(W) \in{\argmin}_{v\in C(X )}\dual{X-Z ,v}.
		\]
		It is not hard to conclude that if $(X,Z)$ is a solution to \cref{eq:AMG-sys-ex}, then it is also a solution to \cref{eq:AMG-sys} and \cref{eq:AMG-sys'}. 
		
		Following \cite[Section 3]{Sonntag2024a}, under the assumption $f_j\in\mathcal F_{L_j}^1(\R^n)$ for all $1\leq j\leq m$, we can show that the differential inclusion  \cref{eq:AMG-sys-ex} exists an absolutely continuous solution $(X,Z)$ over $(0,\infty)$, which satisfies \cref{eq:AMG-sys,eq:AMG-sys'} for almost all $t>0$. In addition, in proper sense (cf.\cite[Definition 3.6]{Sonntag2024a}), $X$ is a solution to \cref{eq:AMG}.
		\subsection{Lyapunov analysis}
		Recall that $f(x;z):=\min_{1\leq j\leq m} \left[f_j(x)-f_j(z)\right]$. Let  $(X,Z)$ be a solution to \cref{eq:AMG-sys-ex}. We introduce the Lyapunov function
		\begin{equation}\label{eq:L-mnag}	
			\mathcal E(t;z) := f(X(t) ;z)+\frac{\gamma(t) }{2}\nm{Z(t) -z}^2,\quad\,t>0,
		\end{equation}
		where $z\in\R^n$.
		To prove the exponential decay, we have to take the time derivative with respect to the nonsmooth function $f(X(t);z)$. This is promised by the following lemma; see \cite[Lemma 4.12]{Sonntag2024a}.
		\begin{lem}\label{lem:df-x-z}
			Let  $(X,Z)$ be a solution to \cref{eq:AMG-sys-ex}, then $f(X;z)$ is differentiable for almost all $t>0$. In particular, for almost all $t>0$, there exists $1\leq i_0\leq m$ such that
			\[
			\frac{\dd }{\dd t}f(X(t) ;z) = \dual{\nabla f_{i_0}(X(t) ),X'(t)}.
			\]
		\end{lem} 
		
		Since $(X,Z)$ satisfies \cref{eq:AMG-sys'}, it follows that
		\[
		\dual{\nabla f_j(X)-\proj_{C(X)}(W)	, (\mu+\alpha+\beta)X' }\leq 0,
		\]
		for all $1\leq j\leq m$, which implies 
		\begin{equation}\label{eq:key-var}
			\dual{\nabla f_j(X)-\proj_{C(X)}(W)	, Z-X }\leq 0.
		\end{equation}
		This yields the exponential decay.
		\begin{thm}\label{thm:rate-L-mop-hnag}
			Let  $(X,Z)$ be a solution to \cref{eq:AMG-sys-ex}, then 
			\[
			\frac{\dd }{\dd t}	\mathcal E(t;z)
			\leq -\mathcal E(t;z)-\frac{\mu}{2}\nm{Z(t)-X(t)}^2,
			\]
			for almost all $t>0$. This implies 
			\begin{equation}\label{eq:exp-Et}
				f(X(t) ;z)+\frac{\gamma(t) }{2}\nm{Z(t) -z}^2\leq e^{-t}\left(f(x_0 ;z)+\frac{\gamma_0 }{2}\nm{x_0+x_1 -z}^2\right),
			\end{equation}
			for all $t>0$.
		\end{thm}
		\begin{proof} 
			Thanks to \cref{lem:df-x-z}, for almost all $t>0$, there exists $1\leq i_0\leq m$ such that 
			\[
			\begin{aligned}
				\frac{\dd }{\dd t}	\mathcal E(t;z) = {}& \frac{\gamma'}{2}\nm{Z-z}^2+\gamma \dual{Z',Z-z}+\dual{\nabla f_{i_0}(X),X'} \\
				= {}& \frac{\mu-\gamma}{2}\nm{Z-z}^2+\gamma \dual{Z',Z-z}+\dual{\nabla f_{i_0}(X),X'} .
			\end{aligned}
			\]		
			Since $(X,Z)$ is a solution to \cref{eq:AMG-sys'}, we obtain
			\[\small
			\begin{aligned}
				\frac{\dd }{\dd t}	\mathcal E(t;z)
				= {}& \frac{\mu-\gamma}{2}\nm{Z-z}^2+ \dual{\mu(X-Z) -\proj_{C(X )}(W ),Z-z}
				+\dual{\nabla f_{i_0}(X),Z-X}\\
				= {}&\frac{\mu}{2}\nm{X-z}^2-\frac{\gamma}{2}\nm{Z-z}^2-\frac{\mu}{2}\nm{Z-X}^2\\
				{}&\quad+\dual{\nabla f_{i_0}(X),Z-X}-\dual{\proj_{C(X )}(W ),Z-z}\\
				= {}&\frac{\mu}{2}\nm{X-z}^2-\frac{\gamma}{2}\nm{Z-z}^2-\frac{\mu}{2}\nm{Z-X}^2-\dual{\proj_{C(X )}(W ),X-z}\\
				{}&\quad+\dual{\nabla f_{i_0}(X)-\proj_{C(X )}(W ),Z-X},
			\end{aligned}
			\]		
			where in the second line we used the identity
			\begin{equation}\label{eq:XYZ}
				2\left\langle X-Z,Z-z\right\rangle=\left\|X-z\right\|^{2}-\left\|Z-z\right\|^{2}-\left\|X-Z\right\|^{2}.
			\end{equation}
			In view of \cref{eq:key-var}, the last line is bounded above by $0$. Assume $\proj_{C(X )}(W ) = \sum_{j=1}^{m}\lambda_j\nabla f_j(X)$ with some $\lambda\in\Delta_m$. Then it holds that 
			\[
			\begin{aligned}
				{}&		\dual{\proj_{C(X )}(W ),X-z}= \sum_{j=1}^{m}\lambda_j	\dual{\nabla f_j(X),X-z}\\
				\geq{}&\sum_{j=1}^{m}\lambda_j	 \left(f_j(X)-f_j(z)+\frac{\mu_j}{2}\nm{X-z}^2\right)\geq\min_{1\leq j\leq m}\left[f_j(X)-f_j(z)\right]+\frac{\mu}{2}\nm{X-z}^2.
			\end{aligned}.
			\]
			Consequently, we arrive at the estimate
			\[
			\begin{aligned}
				\frac{\dd }{\dd t}	\mathcal E(t;z)
				\leq  {}&-\frac{\gamma}{2}\nm{Z-z}^2-\frac{\mu}{2}\nm{Z-X}^2-\min_{1\leq j\leq m}\left[f_j(X)-f_j(z)\right]\\
				={}&-\mathcal E(t;z)-\frac{\mu}{2}\nm{Z-X}^2.
			\end{aligned}
			\]		
			This completes the proof of this theorem.
		\end{proof}
		\subsection{Convergence rate of the merit function} 
		Let  $(X,Z)$ be a solution to \cref{eq:AMG-sys-ex}.
		We shall prove the convergence rate of the merit function $u_0(X(t))=\sup_{z\in\R^n}f(X(t);z)$. 
		Since $X$ satisfies \cref{eq:AMG}, it follows that
		\[
		\dual{\nabla f_j(X)-\proj_{C(X)}(-\gamma X''),(\mu+\gamma)X'}\leq 0,
		\]
		for all $1\leq j\leq m$. This gives
		\[
		\dual{\nabla f_j(X)+\gamma X''+(\mu+\gamma)X',X'}\leq 0,
		\]
		which implies further that
		\[
		\frac{\dd}{\dd t}\left[f_j(X(t))+\frac{\gamma(t)}{2}\nm{X'(t)}^2\right]\leq -\frac{\mu+3\gamma(t)}{2}\nm{X'(t)}^2,
		\]
		for almost all $t>0$. From this we conclude the energy inequality, which together with \cref{lem:u0-,eq:exp-Et} leads to the desired estimate.
		\begin{lem}
			\label{lem:energy-ineq}
			Let  $(X,Z)$ be a solution to \cref{eq:AMG-sys-ex}. Then for any $1\leq j\leq m$,  the energy inequality holds true
			\begin{equation}\label{eq:energy-ineq}
				f_j(X(t))+\frac{\gamma(t)}{2}\nm{X'(t)}^2
				+\frac{1}{2}\int_{0}^{t}(\mu+3\gamma(s))\nm{X'(s)}^2\leq f_j(x_0)+\frac{\gamma_0}{2}\nm{x_1}^2,
			\end{equation}
			for all $t>0$. 
		\end{lem}
		\begin{thm}\label{thm:exp-merit}
			Let  $(X,Z)$ be a solution to \cref{eq:AMG-sys-ex}. Then we have the exponential decay rate for the merit function 
			\begin{equation}\label{eq:exp-merit}
				u_0(X(t))
				\leq{} e^{-t}
				\left[u_0(x_0)+\gamma_0 \nm{x_0+x_1}^2+\gamma_0S^2(\alpha) \right],\quad t>0,
			\end{equation}
			where $	S(\alpha)=\sup_{F^*\in F(P_w\cap \mathcal L_F(\alpha))}\inf_{z\in F^{-1}(F^*)}\nm{z}<+\infty$ with $\alpha=F(x_0)+\gamma_0/2\nm{x_1}^2$.
		\end{thm}
		\begin{proof}
			It is clear that $x_0\in\mathcal L_F(\alpha)$.
			From \cref{eq:energy-ineq} we know that $X(t)\in\mathcal L_F(\alpha)$ for all $t>0$. Hence, by \cref{lem:u0-,eq:exp-Et}, we have
			\[
			\begin{aligned}
				u_0(X(t))={}&\sup_{F^*\in F(P_w\cap \mathcal L_F(\alpha))}\inf_{z\in F^{-1}(F^*)}f(X(t);z)
				\leq{}e^{-t}
				\left[u_0(x_0)+\gamma_0\nm{x_0+x_1}^2+\gamma_0S^2(\alpha) \right].
			\end{aligned}
			\]
			This completes the proof.
		\end{proof}

		\section{An IMEX Scheme with a QP Subproblem}
		\label{sec:imex-qp}
		We now consider proper numerical discretization for our continuous model \cref{eq:AMG-sys'}. In this section, we focus on an implicit-explicit (IMEX) scheme, which corresponds to an accelerated multiobjective gradient method with a Quadratic Programming subproblem. 
		\subsection{Numerical discretization}
		\label{sec:imex-qp-num}
		We propose the following IMEX scheme for \cref{eq:AMG-sys'}:
		\begin{subnumcases}{\label{eq:imex-qp}}
			\label{eq:imex-qp-gk1}
			\dfrac{\gamma_{k+1}-\gamma_k}{\tau_k}={}\mu-\gamma_{k+1},\\	
			\label{eq:imex-qp-xk1}
			\dfrac{x_{k+1}-x_k}{\tau_k}={}z_{k+1}-x_{k+1},\\
			\label{eq:imex-qp-vk1}
			\gamma_k\dfrac{z_{k+1}-z_k}{\tau_k}={}\mu(y_k-z_{k+1})-\proj_{C(y_k)}\left(w^{\rm QP}_{k+1}\right),
		\end{subnumcases}
		where $\tau_k>0$ denotes the step size, $y_k = (x_k+\tau_kz_k)/(1+\tau_k)$ and 
		\begin{equation}\label{eq:wk1}
			w^{\rm QP}_{k+1} ={} -\mu(z_{k+1}-y_k)+\gamma_k\dfrac{x_{k+1}-x_k}{\tau_k}-\gamma_k\dfrac{z_{k+1}-z_k}{\tau_k}.
		\end{equation} 
		By \eqref{eq:imex-qp-xk1} and \eqref{eq:imex-qp-vk1}, we have
		\[
		\frac{\gamma_k+\mu\tau_k}{\tau_k}z_{k+1} = \frac{\gamma_kz_k+\mu\tau_ky_{k}}{\tau_{k}}-\proj_{C(y_{k})}\left(w_k-\left(	\frac{\gamma_k+\mu\tau_k}{\tau_k}-\frac{\gamma_k}{1+\tau_k} \right)z_{k+1}\right),
		\]
		where 
		\[
		w_k=	\frac{\gamma_kz_k+\mu\tau_ky_{k}}{\tau_{k}}-\frac{\gamma_kx_k}{1+\tau_k}.
		\]
		Then according to \cref{thm:sub-prob-case1}, it follows that
		\[
		\begin{aligned}
			z_{k+1} ={}&\frac{\gamma_kz_k+\mu\tau_ky_k-\tau_kz_{k}^{\rm QP}}{\gamma_k+\mu\tau_k},\quad  z_{k}^{\rm QP}={}\proj_{C(y_k)}\left(\mu(y_k-x_k)+\frac{\gamma_k(z_k-x_k)}{\tau_k}\right),
		\end{aligned}
		\]
		which requires the computation of the standard projection (Quadratic Programming) onto the convex hull $C(y_k)=\conv{\nabla f_1(y_k),\cdots,\nabla f_m(y_k)}$. 
		
		Below, in \cref{algo:AMG-QP}, based on the IMEX scheme \cref{eq:imex-qp} and the step size relation $L\tau_k^2=\gamma_k(1+\tau_k)$, we present an accelerated multiobjective gradient method with Quadratic Programming (AMG-QP) for solving \cref{eq:min-mobj}. 
		\begin{algorithm}[H]
			\caption{ AMG-QP}
			\begin{algorithmic}[1] 
				\REQUIRE  Problem parameters: $L>0$ and $\mu\geq 0$.\\
				~~~~~~~Initial values: $	x_0,z_0\in\R^n$ and $\gamma_0>0$.\\
				~~~~~~~Maximum iterations: $ K\in\mathbb N$.			
				\FOR{$k=0,1,\ldots,K-1$}
				%			\STATE Choose $\tau_k>0$ such that $L\tau_k^2\leq (1+\tau_k)\gamma_k$.
				\STATE Choose step size  $\tau_k = (\gamma_k+\sqrt{\gamma_k^2+4L\gamma_k})/(2L)$.			
				\STATE Update $\displaystyle \gamma_{k+1} = (\gamma_k+\mu\tau_k)/(1+\tau_k)$.
				\STATE Compute $y_k = (x_k+\tau_kz_k)/(1+\tau_k)$.
				\STATE Compute $z_{k}^{\rm QP}={}\proj_{C(y_k)}\left(\mu(y_k-x_k)+\frac{\gamma_k(z_k-x_k)}{\tau_k}\right)$.
				% %%%%%%%%%%%%%%%%%lambda%%%%%%%%%%%%%%%%%%%%
				%		\STATE Compute $z_{k}^{\rm QP}={}DF(y_k)\lambda_k$ where
				%		\[
				%		\lambda_k\in\mathop{\argmin}\limits_{\lambda\in\Delta_m}\frac{1}{2} \nm{DF(y_k)\lambda-\mu(y_k-x_k)-\frac{\gamma_k(z_k-x_k)}{\tau_k}}^2 .
				%		\] 		
				% %%%%%%%%%%%%%%%%%lambda%%%%%%%%%%%%%%%%%%%%
				\STATE Update $z_{k+1}=(\gamma_k+\mu\tau_k)^{-1}(\gamma_kz_k+\mu\tau_ky_k-\tau_kz_{k}^{\rm QP})$.
				\STATE Update $\displaystyle  x_{k+1}={}(x_k+\tau_kz_{k+1})/(1+\tau_k)$.
				\ENDFOR
				\ENSURE The approximated solution $x_K\in\R^n$.
			\end{algorithmic}
			\label{algo:AMG-QP}
		\end{algorithm}
		\subsection{Contraction property}
		\label{sec:imex-qp-lyap}
		Let $\{x_k,z_k\}_{k=0}^\infty$ be generated by the IMEX scheme \cref{eq:imex-qp} and introduce a discrete analogue to \cref{eq:L-mnag}:
		\begin{equation}\label{eq:Lk-mnag}	
			\mathcal E_k(z) := f(x_k ;z)+\frac{\gamma_k }{2}\nm{z_k-z}^2,\quad k\geq 0,
		\end{equation}
		where $z\in\R^n$. 
		\begin{lem}
			\label{lem:key-id-imex-qp}
			Let $\{x_k,z_k\}_{k=0}^\infty$ be generated by the IMEX scheme \cref{eq:imex-qp}, then we have 
			\begin{equation}\label{eq:key-id-imex-qp}
				\begin{aligned}
					{}&	\dual{\proj_{C(y_{k})}\left(w^{\rm QP}_{k+1}\right),x_{k+1}-x_k} = {}\max_{1\leq j\leq m}	\dual{\nabla f_j(y_{k}) ,  x_{k+1}-x_k},
				\end{aligned}
			\end{equation}
			where $w^{\rm QP}_{k+1}$ is defined by \cref{eq:wk1}.
		\end{lem}
		\begin{proof}
			Observing \eqref{eq:imex-qp-vk1} we have 
			\[
			\begin{aligned}
				\proj_{C(y_{k})}\left(w^{\rm QP}_{k+1}\right) ={}& -\mu(z_{k+1}-y_{k}) -	\gamma_k	\frac{z_{k+1}-z_k}{\tau_{k}},
			\end{aligned}
			\]
			which implies
			\[
			\begin{aligned}
				{}&	\dual{\nabla f_j(y_{k})-	\proj_{C(y_{k})}\left(w^{\rm QP}_{k+1}\right),w^{\rm QP}_{k+1}+\mu(z_{k+1}-y_{k})+	\gamma_k	\frac{z_{k+1}-z_k}{\tau_{k}}}\leq 0,
			\end{aligned}
			\]
			for all $1\leq j\leq m$. Notice that
			\[
			\begin{aligned}
				{}&	w^{\rm QP}_{k+1}+\mu(z_{k+1}-y_{k})+	\gamma_k	\frac{z_{k+1}-z_k}{\tau_{k}}= 
				\gamma_k 	\frac{x_{k+1}-x_k}{\tau_{k}}.
			\end{aligned}
			\]
			Therefore, we obtain
			\begin{equation}\label{eq:key-AMG-qp}
				\dual{\nabla f_j(y_{k})-\proj_{C(y_{k})}\left(w^{\rm QP}_{k+1}\right),x_{k+1}-x_k}\leq 0.
			\end{equation}
			Clearly, this also yields that
			\[
			\begin{aligned}
				\max_{1\leq j\leq m}	\dual{\nabla f_j(y_{k}),x_{k+1}-x_k}
				\leq		{}& \dual{\proj_{C(y_{k})}\left(w^{\rm QP}_{k+1}\right),x_{k+1}-x_k}\\
				\leq {}&\max_{1\leq j\leq m}	\dual{\nabla f_j(y_{k}),x_{k+1}-x_k }.
			\end{aligned}
			\]
			Consequently, we get the identity \cref{eq:key-id-imex-qp} and complete the proof.
		\end{proof}
		\begin{thm}
			\label{thm:conv-imex-qp}	
			Let $\{x_k,z_k\}_{k=0}^\infty$ be generated by the IMEX scheme \cref{eq:imex-qp}.
			If $L\tau_k^2\leq (1+\tau_k)\gamma_k$, then we have the contraction
			\begin{equation}
				\label{eq:contract-Ek-imex-qp}	
				\mathcal{E}_{k+1}(z)-\mathcal{E}_{k}(z)\leq -\tau_k\mathcal{E}_{k+1}(z),\quad k\geq 0,
			\end{equation}
			which implies that 
			\begin{equation}	\label{eq:conv-imex-qp-fk}	
				f(x_k,z)\leq\theta_k\left(f(x_0;z)+\frac{\gamma_0 }{2}\nm{z_0-z}^2\right),\quad\,k\geq 0,
			\end{equation}
			where $\theta_0=1$ and $\theta_k^{-1}=(1+\tau_0)\cdots(1+\tau_{k-1})$ for all $k\geq 1$.
		\end{thm}
		\begin{proof}
			From \cref{eq:Lk-mnag}, we have that
			\[
			\mathcal{E}_{k+1}(z)-\mathcal{E}_{k}(z)=f(x_{k+1};z)-f(x_{k};z)+\frac{\gamma_{k+1}}{2}\left\|z_{k+1}-z\right\|^{2}-\frac{\gamma_{k}}{2}\left\|z_{k}-z\right\|^{2}.
			\]
			Applying  \cref{eq:XYZ} and \eqref{eq:imex-qp-gk1} gives
			\[
			\begin{aligned}
				{}&	\frac{\gamma_{k+1}}{2}\left\|z_{k+1}-z\right\|^{2}-\frac{\gamma_{k}}{2}\left\|z_{k}-z\right\|^{2}\\
				={}&\frac{\gamma_{k+1}-\gamma_k}{2}\left\|z_{k+1}-z\right\|^{2}+\frac{\gamma_{k}}{2}\left(\left\|z_{k+1}-z\right\|^{2}-\left\|z_{k}-z\right\|^{2}\right)\\
				={}&\frac{\mu\tau_{k}-\tau_{k}\gamma_{k+1}}{2}\left\|z_{k+1}-z\right\|^{2}+\gamma_{k}\dual{z_{k+1}-z_{k},z_{k+1}-z} -\frac{\gamma_{k}}{2}\left\|z_{k+1}-z_{k}\right\|^{2}.	
			\end{aligned}
			\]
			By \eqref{eq:imex-qp-vk1}, we expand the cross term as follows
			\[
			\begin{aligned}
				\gamma_{k}\dual{z_{k+1}-z_{k},z_{k+1}-z}
				={}&\dual{\mu\tau_{k}(y_{k}-z_{k+1})-\tau_{k}\proj_{C(y_{k})}\left(w^{\rm QP}_{k+1}\right),z_{k+1}-z}\\
				={}&\frac{\mu\tau_k}{2}\left(\nm{y_{k}-z}^2-\nm{z_{k+1}-z}^2-\nm{y_{k}-z_{k+1}}^2\right)\\
				{}&\quad-\tau_{k}\dual{\proj_{C(y_{k})}\left(w^{\rm QP}_{k+1}\right),z_{k+1}-z}.
			\end{aligned}
			\]
			Thanks to \eqref{eq:imex-qp-xk1} and \cref{lem:key-id-imex-qp}, it holds that
			\[
			\begin{aligned}
				{}&	- \tau_{k}\left\langle\proj_{C(y_{k})}\left(w^{\rm QP}_{k+1}\right),z_{k+1}-z\right\rangle\\
				={}&-\left\langle\proj_{C(y_{k})}\left(w^{\rm QP}_{k+1}\right),x_{k+1}-x_{k}\right\rangle-\tau_{k}\left\langle\proj_{C(y_{k})}\left(w^{\rm QP}_{k+1}\right),x_{k+1}-z\right\rangle\\
				={}&			
				-				\max_{1\leq j \leq m}\dual{\nabla f_j(y_k),x_{k+1}-x_k}-\tau_{k}\left\langle\proj_{C(y_{k})}\left(w^{\rm QP}_{k+1}\right),x_{k+1}-z\right\rangle. 
			\end{aligned}
			\]
			Collecting the above results gives
			\begin{equation}\label{eq:diff-Ek}
				\begin{aligned}
					\mathcal{E}_{k+1}(z)-\mathcal{E}_{k}(z)={}&f(x_{k+1};z)-f(x_{k};z)	-				\max_{1\leq j \leq m}\dual{\nabla f_j(y_k),x_{k+1}-x_k}\\
					{}&\quad+ \frac{\mu\tau_k}{2}\nm{y_{k}-z}^2-\tau_{k}\left\langle\proj_{C(y_{k})}\left(w^{\rm QP}_{k+1}\right),x_{k+1}-z\right\rangle\\
					{}&\quad -\frac{\tau_{k}\gamma_{k+1}}{2}\left\|z_{k+1}-z\right\|^{2} -\frac{\gamma_{k}}{2}\left\|z_{k+1}-z_{k}\right\|^{2}-\frac{\mu\tau_k}{2}\nm{y_{k}-z_{k+1}}^2.
				\end{aligned}
			\end{equation}
			
			Let us focus on the first and the second lines in \cref{eq:diff-Ek}. Notice that by the Lipschitz continuity of the gradients, we have 
			\begin{equation}\label{eq:key-gd}
				-\dual{f_j(y_k),x_{k+1}-y_k}\leq	f_j(y_k)-f_j(x_{k+1})+\frac{L_j}{2}\nm{x_{k+1}-y_k}^2,
			\end{equation}
			and using \cref{lem:gd-lem} gives 
			\[
			\begin{aligned}
				f_j(x_{k+1})-f_j(x_k)\leq {}&\dual{\nabla f_j(y_k),x_{k+1}-x_k}-\frac{\mu_j}{2}\nm{y_k-z_k}^2+\frac{L_j}{2}\nm{x_{k+1}-y_k}^2,\\
				-\dual{\nabla f_j(y_k),x_{k+1}-z}\leq{}&f_j(z)-f_j(x_{k+1})-\frac{\mu_j}{2}\nm{y_k-z}^2+\frac{L_j}{2}\nm{x_{k+1}-y_k}^2.
			\end{aligned}
			\]
			This implies further that
			\[
			\begin{aligned}
				{}&f(x_{k+1};z)-f(x_k;z) 
				\leq 		\max_{1\leq j \leq m}\left[f_j(x_{k+1})-f_j(x_k)\right]\\
				\leq  {}&\max_{1\leq j \leq m}\dual{\nabla f_j(y_k),x_{k+1}-x_k}+\frac{L}{2}\nm{x_{k+1}-y_k}^2-\frac{\mu}{2}\nm{y_k-z_k}^2.
			\end{aligned}
			\]
			Assume $\proj_{C(y_{k})}\left(w^{\rm QP}_{k+1}\right)
			={}\sum_{j=1}^{m}\lambda_j\nabla f_j(y_k)$ with some $\lambda\in\Delta_m$. Then it follows 
			\[
			\begin{aligned}
				{}& -\left\langle\proj_{C(y_{k})}\left(w^{\rm QP}_{k+1}\right),x_{k+1}-z\right\rangle
				={}-\sum_{j=1}^{m}\lambda_j\dual{\nabla f_j(y_k),x_{k+1}-z}\\
				\leq {}&-\min_{1\leq j\leq m}\left[f_j(x_{k+1})-f_j(z)\right]+\frac{L}{2}\nm{x_{k+1}-y_k}^2-\frac{\mu}{2}\nm{y_k-z}^2.
			\end{aligned}
			\]
			Consequently, putting all together leads to
			\[
			\begin{aligned}
				\mathcal{E}_{k+1}(z)-\mathcal{E}_{k}(z)
				\leq &-\tau_k	\mathcal{E}_{k+1}(z)-\frac{\gamma_{k}}{2}\left\|z_{k+1}-z_{k}\right\|^{2}+\frac{L(1+\tau_k)}{2}\nm{x_{k+1}-y_k}^2.
			\end{aligned}
			\]
			In view of the condition $L\tau_k^2\leq \gamma_k(1+\tau_k)$ and the relation (cf.\eqref{eq:imex-qp-xk1})
			\[
			x_{k+1}-y_k = \frac{x_k+\tau_k z_{k+1}}{1+\tau_k}-\frac{x_k+\tau_k z_k}{1+\tau_k} =\frac{ \tau_k(z_{k+1}-z_k)}{1+\tau_k},
			\]
			we 
			obtain \cref{eq:contract-Ek-imex-qp} and completes the proof of this theorem.
		\end{proof}
		\subsection{Convergence rate of the merit function}
		\label{sec:imex-qp-rate}
		The Lyapunov contraction \cref{thm:conv-imex-qp} is not enough for measuring the optimality. To establish the convergence rate of the merit function \cref{eq:merit}, we shall prove the boundness result of the discrete sequence. 
		
		By \cref{eq:key-AMG-qp} we have
		\[
		\dual{\nabla f_j(y_{k})-\proj_{C(y_{k})}\left(w^{\rm QP}_{k+1}\right),x_{k+1}-x_k}\leq 0,
		\]
		for all $1\leq j\leq m$. This together with \cref{lem:gd-lem} gives
		\[
		f_j(x_{k+1})-f_j(x_{k})+\frac{\mu_j}{2}\nm{y_k-x_k}^2-\frac{L_j}{2}\nm{x_{k+1}-y_k}^2\leq 	\dual{\proj_{C(y_{k})}\left(w^{\rm QP}_{k+1}\right),x_{k+1}-x_k}.
		\]
		In view of \cref{eq:imex-qp}, we have 
		\[
		\begin{aligned}
			{}&	\dual{\proj_{C(y_{k})}\left(w^{\rm QP}_{k+1}\right),x_{k+1}-x_k}
			={}	\dual{\mu(y_k-z_{k+1})-	\gamma_k\dfrac{z_{k+1}-z_k}{\tau_k},x_{k+1}-x_k}\\
			={}&\mu\dual{y_k-x_{k+1}-\frac{x_{k+1}-x_k}{\tau_k},x_{k+1}-x_k}\\
			{}&\quad-\frac{\gamma_k}{\tau_k}\dual{x_{k+1}-x_k+\frac{x_{k+1}-x_k}{\tau_k}-\frac{x_{k}-x_{k-1}}{\tau_{k-1}},x_{k+1}-x_k}\\
			={}&\frac{\mu}{2}\nm{y_k-x_k}^2-\frac{\mu}{2}\nm{x_{k+1}-x_k}^2-\frac{\mu}{2}\nm{x_{k+1}-y_k}^2-\frac{\mu+\gamma_k}{\tau_k}\nm{x_{k+1}-x_k}^2\\
			{}&\quad+\frac{\gamma_k}{2}\left(\nm{\frac{x_k-x_{k-1}}{\tau_{k-1}}}^2-
			\nm{\frac{x_{k+1}-x_{k}}{\tau_{k}}}^2-\nm{\frac{x_{k+1}-x_{k}}{\tau_{k}}-\frac{x_k-x_{k-1}}{\tau_{k-1}}}^2\right)\\
			={}&\frac{\mu}{2}\nm{y_k-x_k}^2-\frac{\mu}{2}\nm{x_{k+1}-y_k}^2-\frac{(2+\tau_k)\mu+2\gamma_k}{2\tau_k}\nm{x_{k+1}-x_k}^2\\
			{}&\quad+\frac{\gamma_k}{2}\left(\nm{\frac{x_k-x_{k-1}}{\tau_{k-1}}}^2-
			\nm{\frac{x_{k+1}-x_{k}}{\tau_{k}}}^2-\nm{z_{k+1}-z_k+x_k-x_{k+1}}^2\right).
		\end{aligned}
		\]
		Therefore, plugging this into the previous estimate leads to
		\[
		\begin{aligned}
			{}&f_j(x_{k+1})-f_j(x_{k})+\frac{\gamma_{k+1}}{2}\nm{\frac{x_{k+1}-x_{k}}{\tau_{k}}}^2-\frac{\gamma_k}{2}\nm{\frac{x_k-x_{k-1}}{\tau_{k-1}}}^2\\
			\leq {}&\frac{L_j-\mu}{2}\nm{x_{k+1}-y_k}^2-\frac{(1+\tau_k)\mu+2\gamma_k+\gamma_{k+1}}{2\tau_k}\nm{x_{k+1}-x_k}^2\\
			{}&\quad-\frac{\gamma_k}{2}\nm{z_{k+1}-z_k+x_k-x_{k+1}}^2+\frac{\mu-\mu_j}{2}\nm{y_k-x_k}^2\\
			={}&\frac{(L_j-\mu)\tau_k^2}{2(1+\tau_k)^2}\nm{z_{k+1}-z_k}^2-\frac{(1+\tau_k)\mu+2\gamma_k+\gamma_{k+1}}{2\tau_k}\nm{x_{k+1}-x_k}^2\\
			{}&\quad-\frac{\gamma_k}{2}\nm{z_{k+1}-z_k+x_k-x_{k+1}}^2+\frac{\mu-\mu_j}{2}\nm{y_k-x_k}^2.
		\end{aligned}
		\]
		Since $\mu\leq\mu_j$ and $(L_j-\mu)\tau_k^2\leq L\tau_{k}^2=\gamma_k(1+\tau_k)$, a direct calculation gives
		\[
		\begin{aligned}
			{}&f_j(x_{k+1})-f_j(x_{k})+\frac{\gamma_{k+1}}{2}\nm{\frac{x_{k+1}-x_{k}}{\tau_{k}}}^2-\frac{\gamma_k}{2}\nm{\frac{x_k-x_{k-1}}{\tau_{k-1}}}^2\\
			\leq{}&\frac{\gamma_k}{2(1+\tau_k)}\nm{z_{k+1}-z_k}^2-\frac{\gamma_k}{2}\nm{z_{k+1}-z_k+x_k-x_{k+1}}^2\\
			{}&\quad-\frac{(1+\tau_k)\mu+2\gamma_k+\gamma_{k+1}}{2\tau_k}\nm{x_{k+1}-x_k}^2\\
			={}&-\frac{\gamma_k\tau_k}{2(1+\tau_k)}\nm{z_{k+1}-z_k}^2-\gamma_k\dual{z_{k+1}-z_k,x_k-x_{k+1}}\\
			{}&\quad-\frac{(1+\tau_k)\mu+(2+\tau_k)\gamma_k+\gamma_{k+1}}{2\tau_k}\nm{x_{k+1}-x_k}^2\\
			\leq 	{}&-\frac{(1+\tau_k)\mu+\gamma_k+\gamma_{k+1}}{2\tau_k}\nm{x_{k+1}-x_k}^2.
		\end{aligned}
		\]
		
		From the above, we conclude the following energy inequality, which together with \cref{lem:u0-,eq:contract-Ek-imex-qp} proves the desired estimate of the merit function.
		\begin{lem}
			\label{lem:energy-ineq-dis}
			Let $\{x_k,z_k\}_{k=0}^\infty$ be generated by \cref{algo:AMG-QP}. Then for any $1\leq j\leq m$,  the energy inequality holds true
			\begin{equation}\label{eq:energy-ineq-dis}
				f_j(x_{k})+\frac{\gamma_{k}}{2}\nm{\frac{x_{k}-x_{k-1}}{\tau_{k}}}^2+\sum_{i=1}^{k-1}\frac{(1+\tau_i)\mu+\gamma_i+\gamma_{i+1}}{2\tau_i}\nm{x_{i+1}-x_i}^2\leq C_j(x_0,x_1,\tau_0),
			\end{equation}
			for all $k\geq 1$, where $C_j(x_0,x_1,\tau_0):=f_j(x_1)+\frac{\gamma_0}{2}\nm{(x_1-x_0)/\tau_0}^2$.
		\end{lem}
		\begin{thm}\label{thm:exp-merit-dis}
			Let $\{x_k,z_k\}_{k=0}^\infty$ be generated by \cref{algo:AMG-QP}, then 
			\begin{equation}\label{eq:exp-merit-dis}
				u_0(x_k)
				\leq{}		\left[u_0(x_0)+\gamma_0\nm{z_0}^2+\gamma_0S^2(\alpha)\right]\times  \min\left\{\frac{4L}{\gamma_0k^2},\,\left(1+\sqrt{\min\{\mu,\gamma_0\}/L}\right)^{-k}\right\},
			\end{equation}
			for all $k\geq 1$, where $	S(\alpha)=\sup_{F^*\in F(P_w\cap \mathcal L_F(\alpha))}\inf_{z\in F^{-1}(F^*)}\nm{z}<+\infty$ with $\alpha=F(x_1)+\gamma_0/2\nm{(x_1-x_0)/\tau_0}^2$.
		\end{thm}
		\begin{proof}
			From \cref{eq:energy-ineq-dis} we claim that $\{x_k\}_{k=1}^\infty\subset\mathcal L_F(\alpha)$. Hence, by \cref{lem:u0-,eq:conv-imex-qp-fk}, we have
			\[
			\begin{aligned}
				u_0(x_k)={}&\sup_{F^*\in F(P_w\cap \mathcal L_F(\alpha))}\inf_{z\in F^{-1}(F^*)}f(x_k;z)\\
				\leq{}&\theta_k\sup_{F^*\in F(P_w\cap \mathcal L_F(\alpha))}\inf_{z\in F^{-1}(F^*)}\left(f(x_0;z)+\frac{\gamma_0 }{2}\nm{z_0-z}^2 \right)\\
				\leq{}&\theta_k\left[u_0(x_0)+\gamma_0\nm{z_0}^2+\gamma_0S^2(\alpha)\right].
			\end{aligned}
			\]
			It remains to establish the decay rate of $\theta_k^{-1}=(1+\tau_0)\cdots(1+\tau_{k-1})$. Utilizing \cite[Lemma B2]{luo_differential_2021} gives that 
			\[
			\theta_k\leq \min\left\{\frac{4L}{(2\sqrt{L}+\sqrt{\gamma_0}k)^2},\,\left(1+\sqrt{\min\{\mu,\gamma_0\}/L}\right)^{-k}\right\}.
			\]
			This completes the proof of this theorem.
		\end{proof}

		\section{Backtracking and Adaptive Restart}
		\label{sec:back-res}
		For accelerated methods, the gradient Lipschitz constant $L$ and the strong convexity parameter $\mu$ are important for both practical implementation and theoretical analysis. Although linear rate has been established in \cref{thm:exp-merit-dis} for strongly convex problems, it requires the priori knowledge about $\mu$, which is usually more difficult to estimate than $L$. Otherwise, in most cases, accelerated methods with $\mu=0$ will lead to high oscillation; see the numerical results in \cref{sec:num-res} and \cite[Section 7]{Tanabe2023a}.
		
		In this section, we propose two variants of \cref{algo:AMG-QP} to handle these issues. The first adopts the standard backtracking technique, which provides useful estimate of the local Lipschitz constant. The second uses two adaptive restart strategies, to get rid of the dependence on $\mu$ and reduce the drawback of oscillation.
		\subsection{Backtracking}
		Given $\gamma_0,L>0,\,\mu\geq0$ and $x_0,z_0\in\R^n$ , let 
		\[
		(\tau^+,\gamma^+,x^+,y^+,z^+)=\texttt{AMG-QP}(\mu,L,\gamma_0,x_0,z_0)
		\]
		 be the output of \cref{algo:AMG-QP} with $one$ iteration $K=1$. In \cref{algo:AMG-QP-line}, we give an implementable variant of \cref{algo:AMG-QP} with backtracking. Following \cite[Section 3.1]{Luo2024ai}, under \cref{assum:Lj-muj}, it is not hard to prove that for any $k\in\mathbb N$, the total number of backtracking steps is finite.
		\begin{algorithm}[H]
			\caption{AMG-QP with backtracking}
			\label{algo:AMG-QP-line}
			\begin{algorithmic}[1] 
				\REQUIRE  	  Problem parameters: $M_0>0$ and $\mu\geq 0$.\\
				~~~~~~~Initial values: $	x_0,z_0\in\R^n$ and $\gamma_0>0$.\\		
				\FOR{$k=0,1,\cdots$}
				\STATE Set $i = 0,\,M_{k,0} = M_{k}$.
				\STATE Compute $(\tau^+,\gamma^+,x^+,y^+,z^+)=\texttt{AMG-QP}(\mu,M_{k,i},\gamma_k,x_k,z_k)$. 
				\STATE Compute $\delta_j=f_j(x^+)-f_j(y^+)-\dual{\nabla f_j(y^+),x^+-y^+}$ for $1\leq j\leq m$.
				\WHILE[Backtracking]{$\max\{\delta_j/M_{k,i}:1\leq j\leq m\}>	\frac{1}{2}\nm{x^+-y^+}^2$}			
				\STATE Update $i = i + 1$ and $M_{k,i} = 2^{i}M_{k,0}$.				
				\STATE Compute $(\tau^+,\gamma^+,x^+,y^+,z^+)=\texttt{AMG-QP}(\mu,M_{k,i},\gamma_k,x_k,z_k)$.
				\ENDWHILE		
				\STATE Set $i_k = i,\,\tau_k =  \tau^+,\,M_{k+1} = M_{k,i_k}$.		
				\STATE Update $\displaystyle \gamma_{k+1} = \gamma^+,x_{k+1} =x^+,\,z_{k+1} =z^+$. 
				\ENDFOR
			\end{algorithmic}
		\end{algorithm}
		For any fixed $k\in\mathbb N$, given the current state $\{M_k,\gamma_k,x_k,z_k\}$, the  backtracking procedure returns $\{M_{k+1},\gamma_{k+1},x_{k+1},y_k,z_{k+1}\}$ that satisfies
		\[
		\left\{
		\begin{aligned}
			{}&	\dfrac{\gamma_{k+1}-\gamma_k}{\tau_k}={}\mu-\gamma_{k+1},\\	 
			{}&	\dfrac{x_{k+1}-x_k}{\tau_k}={}z_{k+1}-x_{k+1},\\ 
			{}&	\gamma_k\dfrac{z_{k+1}-z_k}{\tau_k}={}\mu(y_k-z_{k+1})-\proj_{C(y_k)}\left(w^{\rm QP}_{k+1}\right),
		\end{aligned}
		\right.
		\]
		where 
		\[
		\begin{aligned}
			y_k ={}& (x_k+\tau_kz_k)/(1+\tau_k),\\	
			\tau_k = {}&(\gamma_k+\sqrt{\gamma_k^2+4M_{k+1}\gamma_k})/(2M_{k+1}),\\
			w^{\rm QP}_{k+1} ={}& -\mu(z_{k+1}-y_k)+\gamma_k\dfrac{x_{k+1}-x_k}{\tau_k}-\gamma_k\dfrac{z_{k+1}-z_k}{\tau_k}.
		\end{aligned}
		\]
		Additionally, $(x_{k+1},y_k,M_{k+1})$ satisfies a key inequality like \cref{eq:key-gd}:
		\[
		-\dual{f_j(y_k),x_{k+1}-y_k}\leq	f_j(y_k)-f_j(x_{k+1})+\frac{M_{k+1}}{2}\nm{x_{k+1}-y_k}^2,
		\]
		for all $1\leq j\leq m$. Hence, we claim that the Lyapunov contraction (cf.\cref{thm:conv-imex-qp}) and the energy decay (cf.\cref{lem:energy-ineq-dis}) can be proved similarly. Consequently, the fast convergence rate (cf.\cref{thm:exp-merit-dis}) of \cref{algo:AMG-QP-line} follows immediately. For simplicity, we omit the details.
	\subsection{Adaptive restart}
	\label{sec:adp-res}
	To overcome the oscillation of acceleration methods for single objective problems $(m=1)$, the following adaptive restart techniques are well known:
	\begin{itemize}
		%	\item Fixed restart (FixR): \text{O’Donoghue and Cand\`{e}s({\it FoCM}, 2015)}
		%\[
		%\frac{\dd }{\dd t}F(x(t))>0
		%\]	
		\item {\bf Function restart (FR)} by O'Donoghue and Cand\`{e}s \cite[Section 3.2]{ODonoghue2015}: 
		%	$F(x_{k+1})>F(x_k)$.
		\begin{equation}\label{eq:FR}
			\tag{FR}
			%\[
			\frac{\dd }{\dd t}F(x(t))>0\quad\Longleftrightarrow\quad
			F(x_{k+1})>F(x_k).
		\end{equation}
		%\]
		\vskip0.1cm
		\item {\bf Gradient restart (GR)} by O'Donoghue and Cand\`{e}s \cite[Section 3.2]{ODonoghue2015}: 
		%	$\dual{\nabla F(x_k),x_{k+1}-x_k}>0$.
		\begin{equation}\label{eq:GR}
			\tag{GR}
			%\[
			\frac{\dd }{\dd t}F(x(t))=	\dual{\nabla F(x(t)),x'(t)}>0
			\quad\approx\quad
			\dual{\nabla F(y_k),x_{k+1}-x_{k}}>0,
		\end{equation}
		%\]
		where $y_k$ is the extrapolation term.
		\vskip0.1cm
		\item {\bf Speed restart (SR)} by Su et al. \cite[Section 5]{su_dierential_2016}:
		%	 $\nm{x_{k+1}-x_k}<\nm{x_{k}-x_{k-1}}$.
		\begin{equation}\label{eq:SR}
			\tag{SR}
			%		\frac{\dd}{\dd t} \nm{x'(t)}^2<0
			%	\quad\Longleftrightarrow\quad
			\frac{\dd}{\dd t}\nm{x'(t)}^2<0\quad\Longleftrightarrow\quad 	\nm{x_{k+1}-x_k}<\nm{x_{k}-x_{k-1}}.
		\end{equation}
	\end{itemize}
	For multiobjective problems, except \cref{eq:SR}, it seems not easy to use \cref{eq:FR,eq:GR}, since there are more than one objectives and gradients. However, as mentioned in \cite{Maulen2023a}, \cref{eq:SR} is not much effective and requires proper warm start. Motivated by the efficient {\bf Gradient norm restart (GnR)} \cite[Page 24]{Luoode2opt2024} 
	\begin{equation}\label{eq:GnR}
		\tag{GnR}
		\frac{\dd}{\dd t}\nm{\nabla F(x(t))}^2>0
		\quad\Longleftrightarrow\quad \nm{\nabla F(x_{k+1})}>\nm{\nabla F(x_{k})},
	\end{equation}
	we propose a novel {\bf Residual restart (ResR)} technique:
	\begin{equation}\label{eq:ResR}
		\tag{ResR}
		\frac{\dd}{\dd t}\nm{\proj_{C(x(t))}(0)}^2>0
		\quad\Longleftrightarrow\quad 
		\nm{\proj_{C(x_{k+1})}(0)}>	\nm{\proj_{C(x_{k})}(0)}.
	\end{equation}
	In other words, we restart the algorithm whenever the KKT residual $\|\proj_{C(x_{k})}(0)\|$ is no longer decreasing.
	
	A combination of \cref{algo:AMG-QP-line} with \cref{eq:SR,eq:ResR} is summarized in \cref{algo:AMG-QP-line-res}. Numerical results in \cref{sec:num} show that both two adaptive      restart schemes work very well and \cref{eq:ResR} is more efficient than \cref{eq:SR}. The linear convergence (competitive with using strong convexity parameter) is observed but the rigorous justification deserves further investigations. For \cref{eq:SR}, we see existing linear rate  of function value in the continuous level for single objective problems \cite[Section 5]{su_dierential_2016} and \cite{Maulen2023a}, but it is unclear for the multiobjective case. As for \cref{eq:GnR,eq:ResR}, the theoretical linear rate also remains open and the analysis is much more difficult since there involves the Hessian information.
	\begin{algorithm}
		\caption{AMG-QP with backtracking and restart}
		\label{algo:AMG-QP-line-res}
		\begin{algorithmic}[1] 
			\REQUIRE  	  Problem parameters: $M_0>0$ and $\mu= 0$.\\
			%		~~~~~~~Line search parameters: $\rho_{\! u}>1,\,\rho_{_{\!d}}\geq 1$.\\		
			~~~~~~~Initial values: $	x_0,z_0\in\R^n$ a
			nd $\gamma_0>0$.\\		
			%		~~~~~~~Maximum iterations: $ K\in\mathbb N$.	\\
			\FOR{$k=0,1,\cdots$}
			\STATE Set $i = 0,\,M_{k,0} = M_{k}$.
			\STATE Compute $(\tau^+,\gamma^+,x^+,y^+,z^+)=\texttt{AMG-QP}(\mu,M_{k,i},\gamma_k,x_k,z_k)$.
			\STATE Compute $\delta_j=f_j(x^+)-f_j(y^+)-\dual{\nabla f_j(y^+),x^+-y^+}$ for $1\leq j\leq m$.\label{algo:sub-uapd}
			\WHILE[Backtracking]{$\max\{\delta_j/M_{k,i}:1\leq j\leq m\}>	\frac{1}{2}\nm{x^+-y^+}^2$}		
			\label{algo:line-search}		
			\STATE Update $i = i + 1$ and $M_{k,i} = 2^{i} M_{k,0}$.				\label{algo:Mki}		
			\STATE Compute $(\tau^+,\gamma^+,x^+,y^+,z^+)=\texttt{AMG-QP}(\mu,M_{k,i},\gamma_k,x_k,z_k)$.
			\label{algo:sub-uapd-2}
			\ENDWHILE
			\label{algo:line-search-end}		
			\STATE Set $i_k = i,\,\tau_k =  \tau^+,\,M_{k+1} = M_{k,i_k}$.		\label{algo:Mk1}		
			\STATE Update $\displaystyle \gamma_{k+1} = \gamma^+,x_{k+1} =x^+,\,z_{k+1} =z^+$.	
			%		\label{algo:lk1}
			\IF[Restart]{either \cref{eq:SR} or \cref{eq:ResR} is satisfied}		
			\label{algo:res}		
			\STATE Reset $\gamma_{k+1} = \gamma_0,\,x_{k+1}=x_k,\,z_{k+1}=x_k$.
			\ENDIF	
			\ENDFOR
		\end{algorithmic}
	\end{algorithm}

		\section{Numerical Tests}
		\label{sec:num}
		\subsection{Setup}
		We compare \cref{algo:AMG-QP-line,algo:AMG-QP-line-res} with three multiobjective gradient methods:
		\begin{itemize}
			\item Steepest descent (SD) \cite{fliege_steepest_2000}.
			%	: given $x_0\in\R^n$, compute
			%\[
			%\begin{aligned}
			%		x^{k+1} =			{}& \mathop{\argmin}\limits_{z\in\R^n}  \max _{1\leq j\leq m}\!\left\{\dual{\nabla f_j(x_k), z-x_k}+\frac{1}{2\tau}\|z-x_k\|^2\right\}\\
			%	={}&x_k-\tau \proj_{C(x_k)}\left( 0\right)
			%	={}x_k-\tau DF(x_k)\lambda_k^*,
			%\end{aligned}
			%\]
			%	where $\tau^{-1}\geq  L:=\max_{1\leq j\leq m}\{L_j\}$ and 
			%		\[
			%	\lambda_k^*\in\mathop{\argmin}\limits_{\lambda\in\Delta_m}\frac1{2}\nm{DF(y_k)\lambda}^2.
			%	\]
			\item Accelerated gradient (AccG ) \cite{Sonntag2024}.
			%	: given $x_0=x_{-1}\in\R^n$, compute
			%	\[
			%	\left\{
			%	\begin{aligned}
				%		y_{k} = {}&x_{k}+\frac{k}{k+3}(x_{k}-x_{k-1}),\\
				%			x_{k+1}  
				%			={}&y_k-\tau\proj_{C(y_k)}\left( \frac{k}{k+3}\cdot\frac{x_{k}-x_{k-1}}{\tau}\right)=y_{k}-\tau DF(y_k)\lambda_{k}^*,
				%		\end{aligned}
			%	\right.	
			%	\]
			%	where $\tau^{-1}\geq  L:=\max_{1\leq j\leq m}\{L_j\}$ and 
			%	\[
			%	\lambda_{k}^*\in\mathop{\argmin}\limits_{\lambda\in\Delta_m}\left\{-\frac{k}{k+3}\dual{DF(y_k)\lambda,\frac{x_{k}-x_{k-1}}{\tau}}+\frac1{2}\nm{DF(y_k)\lambda}^2\right\}.
			%	\]
			\item Accelerated proximal gradient (APG) \cite{Tanabe2023a}.
			%	: 	given $\theta_0=1$ and $x_0=y_0\in\R^n$, compute
			%\[
			%	\left\{
			%\begin{aligned}
			%\theta_{k+1}^{-1}=	{}&	\sqrt{\theta_k^{-2}+1 / 4}+1 / 2,\\
			%	x^{k+1} =	{}&   \mathop{\argmin}\limits_{z\in\R^n} \max _{1\leq j\leq m}\!\left\{\left[\left\langle\nabla f_j(y_k), z-y_k\right\rangle+f_j(y_k)-f_j(x_k)\right]+\frac{1}{2\tau}\|z-y_k\|^2\right\}\\
			%	={}&y_k-\tau DF(y_k)\lambda_k^*,\\
			%	y^{k+1} =	{}&	x^{k+1}+ \theta_{k+1}\left(\theta_k^{-1}-1\right) \left(x^{k+1}-x^{k}\right),
			%\end{aligned}
			%\right.	
			%\]
			%	where $\tau^{-1}\geq  L:=\max_{1\leq j\leq m}\{L_j\}$ and 
			%	\[
			%	\lambda_k^*\in\mathop{\argmin}\limits_{\lambda\in\Delta_m}\left\{\dual{\lambda,\frac{F(x_k)-F(y_k)}{\tau}}+\frac1{2}\nm{DF(y_k)\lambda}^2\right\}.
			%	\]
			%	\item Accelerated gradient  method without quadratic subproblems (AccG w\verb|\|o Q) \cite{Sonntag2024}: given $x_0=x_{-1}\in\R^n$, compute
			%		\[
			%	\left\{
			%	\begin{aligned}
				%{}&		y_{k} = x_{k}+\frac{k}{k+3}(x_{k}-x_{k-1}),\\
				%{}&		j_k\in \argmax_{1\leq j\leq m}\dual{\nabla f_j(y_k),y_k-x_{k}},\\
				%	{}&	x_{k+1}  
				%		=y_{k}-\tau \nabla f_{j_k}(y_k),
				%	\end{aligned}
			%	\right.	
			%	\]
			%	where $\tau^{-1}\geq  L:=\max_{1\leq j\leq m}\{L_j\}$.
		\end{itemize}
		According to \cite[Theorems 5.2 and 5.3]{tanabe_convergence_2023}, SD converges with the sublinear rate $L/k$ for convex case $\mu=0$ and the linear rate $(1-\mu/L)^k$ for strongly convex case $\mu>0$. Both APG and AccG work only for convex objectives with the same theoretical rate  $L/k^2$. To distinguish the strongly convex case and the convex case, we use AMG-QP($\mu=0$) and AMG-QP($\mu>0$) to denote \cref{algo:AMG-QP-line} for $\mu=0$ and $\mu>0$ respectively. Similarly, AMG-QP($\mu=0$)-SR and AMG-QP($\mu=0$)-ResR stand for \cref{algo:AMG-QP-line-res} with the speed restart \cref{eq:SR} and the residual restart \cref{eq:ResR}.
		
		As mentioned in \cite[Section 7.2]{Sonntag2024} and \cite[Remark 4.1]{Tanabe2023a}, we can apply the backtracking technique to AccG and APG for estimating the local Lipschitz constant. Note also that all the methods (including ours) involve the similar subproblem
		\[
		x^{k+1} =	{}  \mathop{\argmin}\limits_{z\in\R^n} \max _{1\leq j\leq m}\!\left\{ \left\langle\nabla f_j(y_k), z-y_k\right\rangle-p_j^k +\frac{1}{2\tau}\|z-y_k\|^2\right\},
		\]
		where $\tau>0,\,y_k\in\R^n$ and $P^k = \left[p^k_1,\cdots,p_m^k\right]^\top\in\R^m$ are given.
		The solution is given by $x_{k+1} =y_k-\tau DF(y_k)\lambda_k^*$, where $\lambda_k^*$ can be obtained from a Quadratic Programming 
		\[
		\lambda_k^*\in\mathop{\argmin}\limits_{\lambda\in\Delta_m}\left\{\langle \lambda, P^k\rangle+\frac\tau{2}\nm{DF(y_k)\lambda}^2\right\}.      
		\]
		For all algorithms, we use the same setting (unless other specified):
		\begin{itemize}
			\item Problem dimension $n = 100$  
			\item Number of sampling points $N = 100$	
			\item Initial backtracking parameter $M_0= 10$ 				
			\item Initial points are uniformly sampled from $[-2,2]^n$ 
		\end{itemize}  
		\subsection{Computational results}
		\label{sec:num-res}
		\noindent{\bf Example 1.} Consider the log-sum-exp objectives:
		\begin{equation}
			\label{eq:ex1}
			f_j(x) =\frac{\delta}{2}\nm{x}^2+\ln \sum_{i=1}^p\exp\left( \langle a_i^j,x\rangle-b_i^j\right) ,\quad j = 1,2,3,
		\end{equation}
		where $\delta\geq0,\, a_i^j\in\R^{n}$ and $b^j_i \in\R$ for $1\leq i\leq p$. We take $p=100,\delta = 0.05$ and generate $a_i^j$ and $b_i^j$ uniformly from $[-1,1]$. In \cref{fig:test_LSE_PFk}, we show the approximate Pareto front at the iteration $k=25$. In \cref{fig:test-lse-resk-iter,fig:test-lse-resk-time}, we report the convergence behavior of the KKT residual $\|\proj_{C(x_k)}(0)\|$ and the iterate gap $\nm{x_{k+1}-x_k}$ with respect to the iteration step and the iteration time (in second). It can be observed that: (i) our AMG-QP($\mu=0$) performs very similarly to APG and AccG; (ii) utilizing the strong convexity parameter $\mu=\delta$ does improve the performance much better than the rest methods; (iii) the adaptive restart technique works very well and our residual restart \cref{eq:ResR} outperforms the speed restart \cref{eq:SR}. 
		\begin{figure}[H]
			\centering
			%						\raggedright
			%					\begin{minipage}{9cm}
				\includegraphics[width=0.5\textwidth]{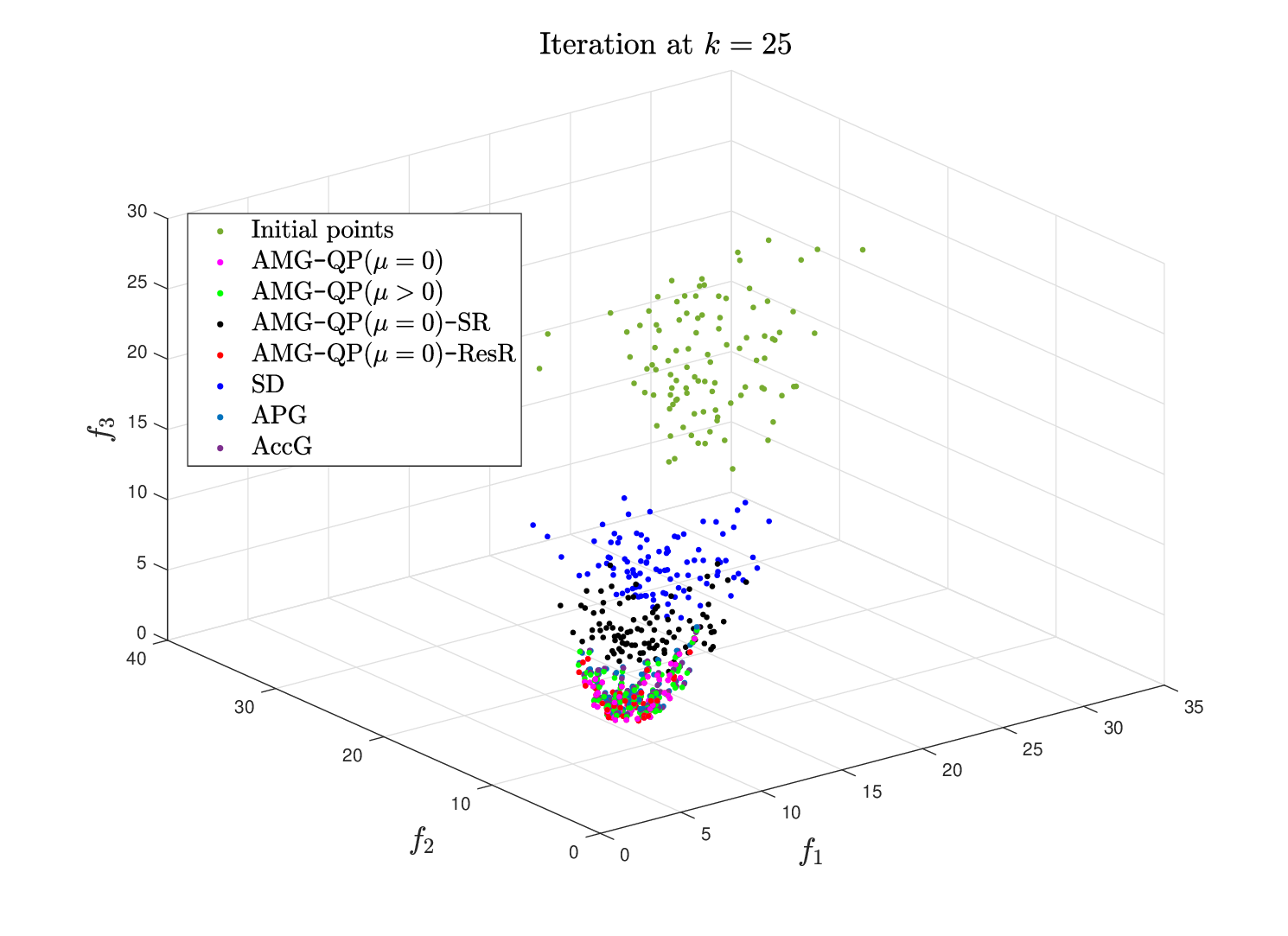}
				%							\end{minipage} 
			\caption{The approximate Pareto front at the iteration $k = 25$, with $N=100$ initial sample points in $[-2,2]^n$ for the first example \cref{eq:ex1}.}
			\label{fig:test_LSE_PFk}
		\end{figure}
		\begin{figure}[H]
			\centering
			\includegraphics[width=0.7\textwidth]{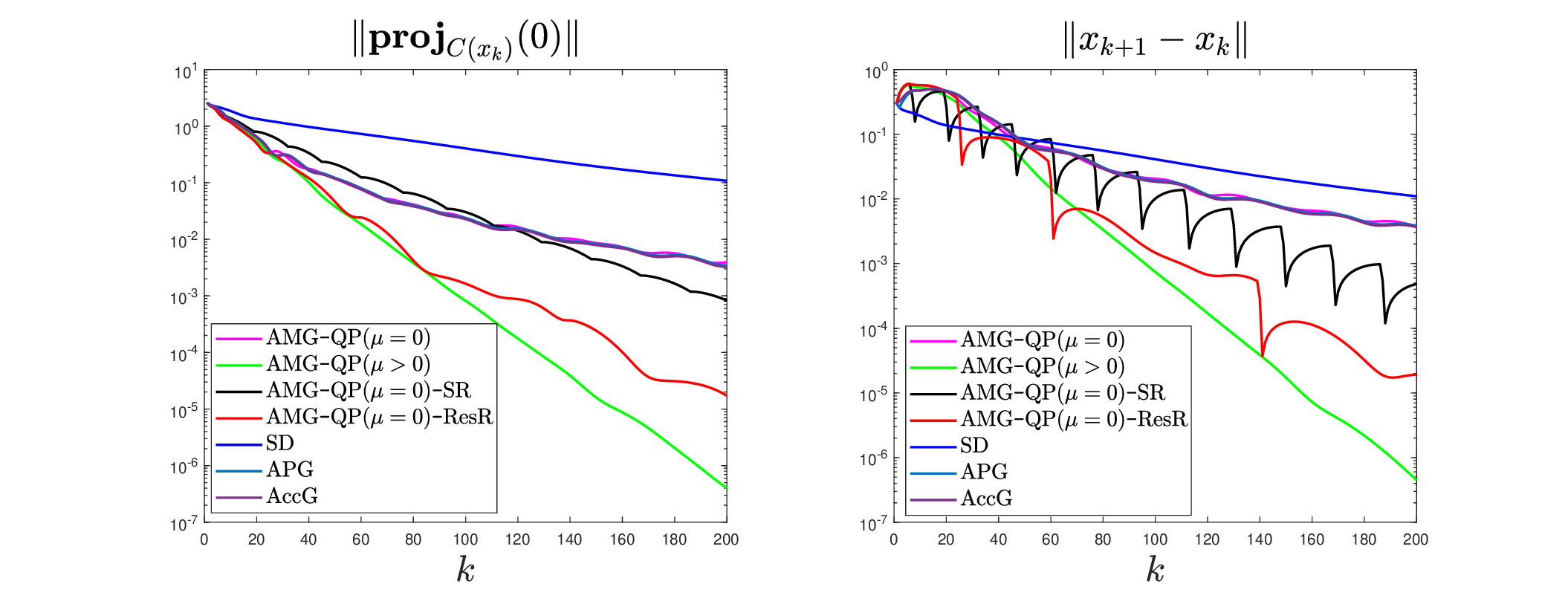}
			\caption{The KKT residual $\|\proj_{C(x_k)}(0)\|$ and the iterate gap $\nm{x_{k+1}-x_k}$ $v.s.$ the iteration step $k$, with one initial sample point $x_0\in[-2,2]^n$ for the first example \cref{eq:ex1}.}
			\label{fig:test-lse-resk-iter}
		\end{figure} 
		\begin{figure}[H]
			\centering 
			\includegraphics[width=0.7\textwidth]{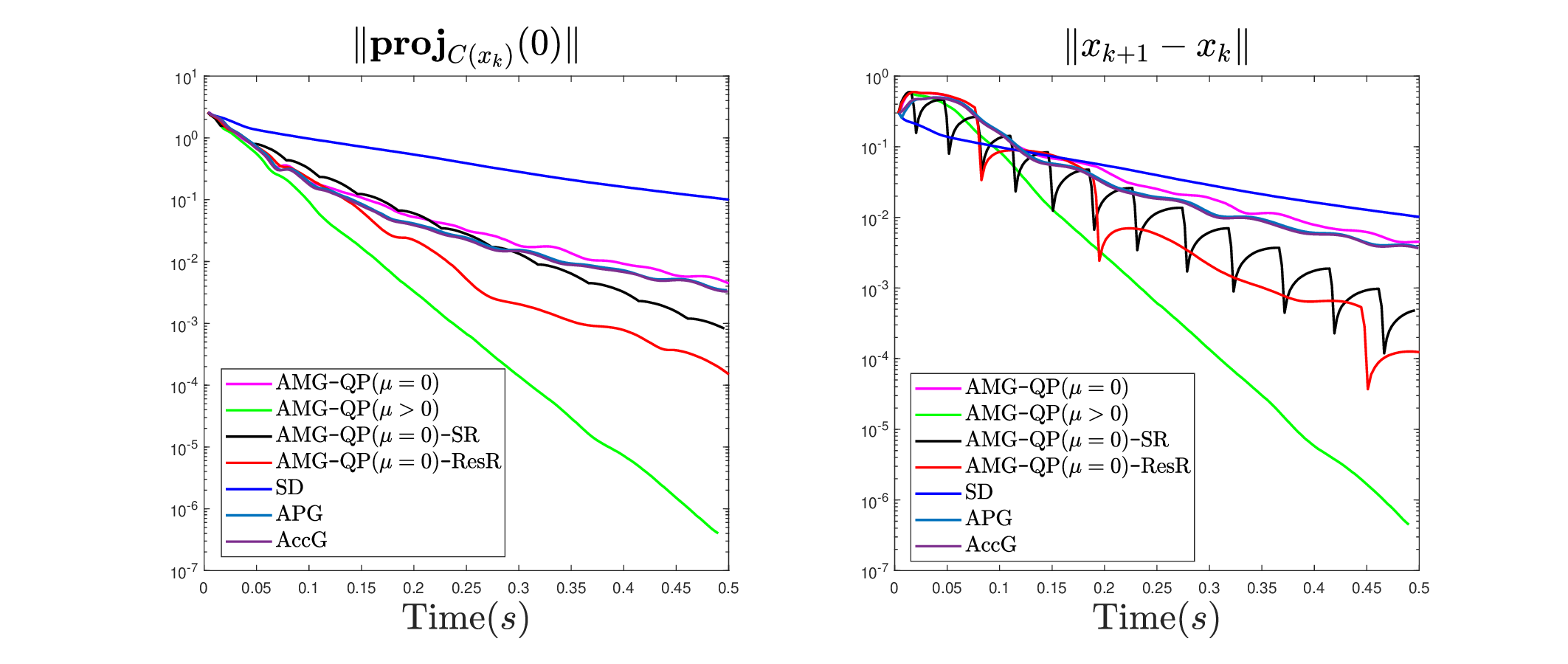} 
			\caption{The KKT residual $\|\proj_{C(x_k)}(0)\|$ and the iterate gap $\nm{x_{k+1}-x_k}$ $v.s.$ the iteration time $s$, with one initial sample point $x_0\in[-2,2]^n$ for the first example \cref{eq:ex1}.}
			\label{fig:test-lse-resk-time}
		\end{figure}
		\noindent{\bf Example 2.} Then let us consider the second example
		\begin{equation}
			\label{eq:ex2}
			f_j(x) = \frac{\delta}{2}\nm{x}^2+\frac{1}{2}\nm{A^jx-b^j}^2,\quad j = 1,2,
		\end{equation}
		where $\delta\geq0,\,A^j  \in\R^{n\times p}$ and $b^j \in\R^p$. Again, we take $p=100,\delta = 0.05$ and generate $A^j$ and $b^j$ uniformly from $[0,1]$. From \cref{fig:test-lsr-resk-iter,fig:test-lsr-resk-time}, we observe that and our residual restart \cref{eq:ResR} beats the speed restart \cref{eq:SR} and is even competitive with AMG-QP($\mu>0$). The rest three accelerated methods: AMG-QP($\mu=0$), APG and AccG, outperform SD but imply high oscillation phenomena for both KKT residual and iterate gap.
		\begin{figure}[H]
			\centering 
			\includegraphics[width=0.7\textwidth]{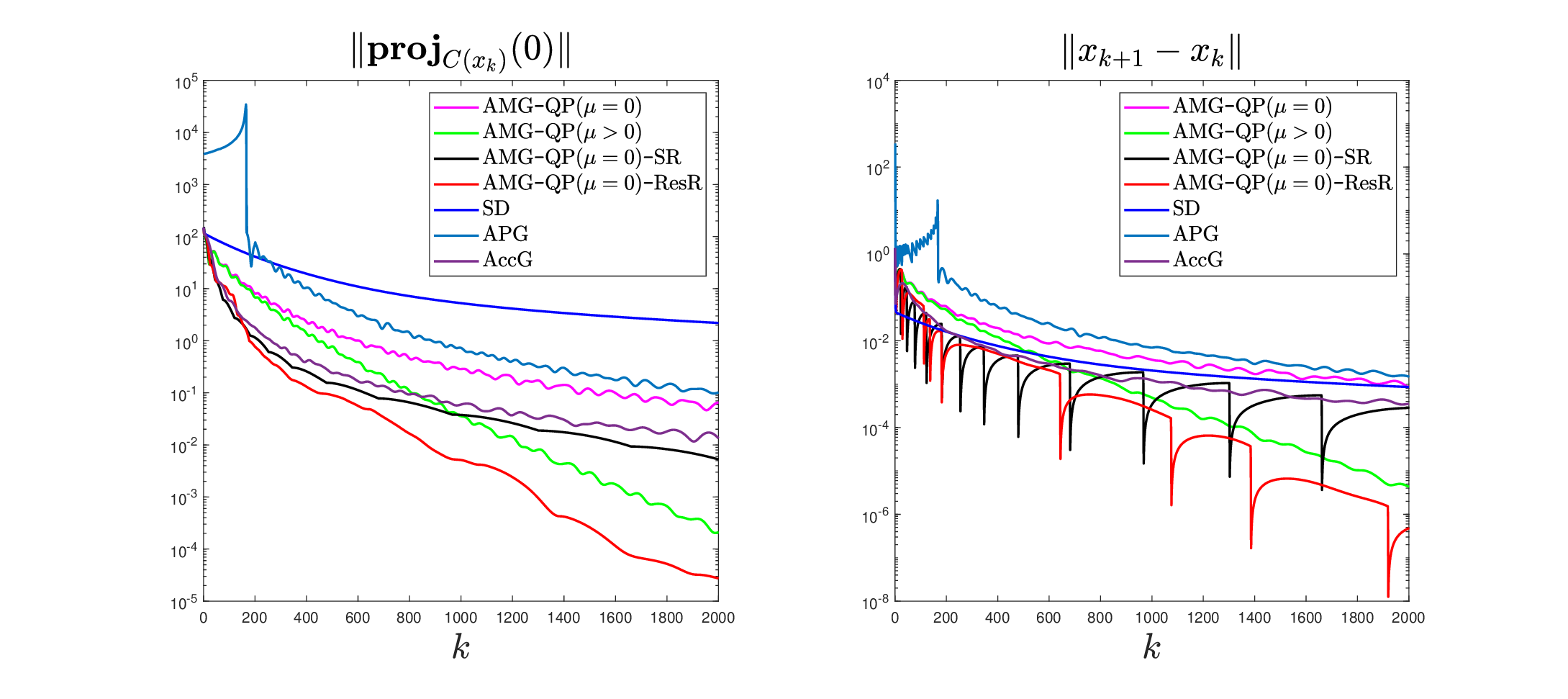} 
			\caption{The KKT residual $\|\proj_{C(x_k)}(0)\|$ and the iterate gap $\nm{x_{k+1}-x_k}$ $v.s.$ the iteration step $k$, with one initial sample point $x_0\in[-2,2]^n$ for the second example \cref{eq:ex2}.}
			\label{fig:test-lsr-resk-iter}
		\end{figure} 
		\begin{figure}[H]
			\centering 
			\includegraphics[width=0.7\textwidth]{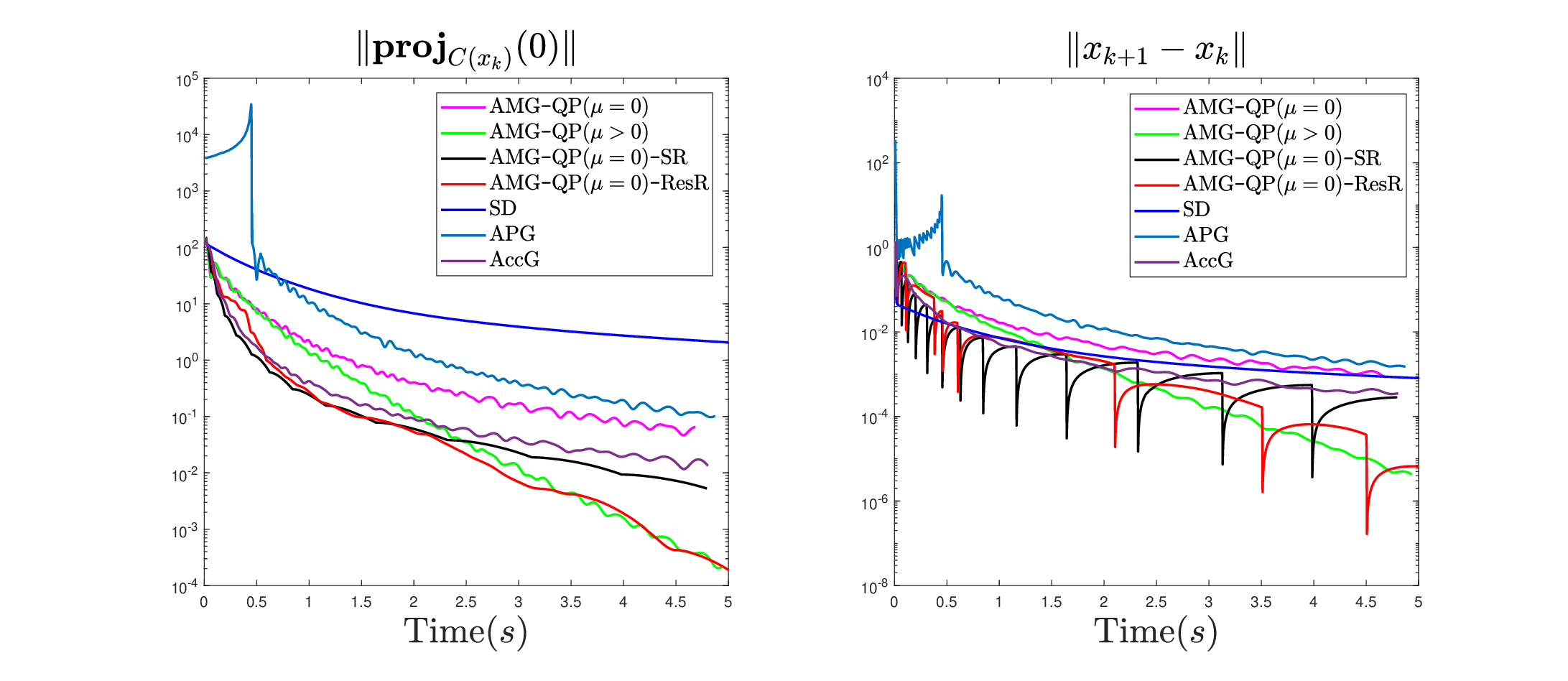} 
			\caption{The KKT residual $\|\proj_{C(x_k)}(0)\|$ and the iterate gap $\nm{x_{k+1}-x_k}$ $v.s.$ the iteration time $s$, with one initial sample point $x_0\in[-2,2]^n$ for the second example \cref{eq:ex2}.}
			\label{fig:test-lsr-resk-time}
		\end{figure}
		\noindent{\bf Example 3.} Lastly, we consider a nonconvex example \cite{Sonntag2024}
		\begin{equation}
			\label{eq:ex3}
			\begin{aligned}
				f_1(x) ={}& \frac{1}{2}\left(\sqrt{1+\snm{a^\top_1x}^2}+\sqrt{1+\snm{a^\top_2x}^2}+a^\top_2x\right)+\exp\left(-\snm{a^\top_2x}^2\right),\\
				f_2(x) ={}& \frac{1}{2}\left(\sqrt{1+\snm{a^\top_1x}^2}+\sqrt{1+\snm{a^\top_2x}^2}-a^\top_2x\right)+\exp\left(-\snm{a^\top_2x}^2\right),
			\end{aligned}
		\end{equation}
		where $a_1$ and $a_2$ are generated uniformly from $[0,1]^n$. In view of \cref{fig:test_NC_Resk_iter,fig:test_NC_Resk_time}, we observe fast convergence of our \cref{eq:ResR} scheme. For this problem, we also plot the approximate Pareto front of all methods in \cref{fig:test_NC_PF}. 
		
		As a summary, we conclude that our AMG-QP($\mu=0$) is very close to APG and AccG for solving convex problems. For strongly convex problems, AMG-QP($\mu>0$) converges with fast linear convergence. Moreover, applying our residual restart \cref{eq:ResR} to AMG-QP($\mu=0$) improves the practical performance significantly.
		\begin{figure}[H]
			\centering 
			\includegraphics[width=0.78\textwidth]{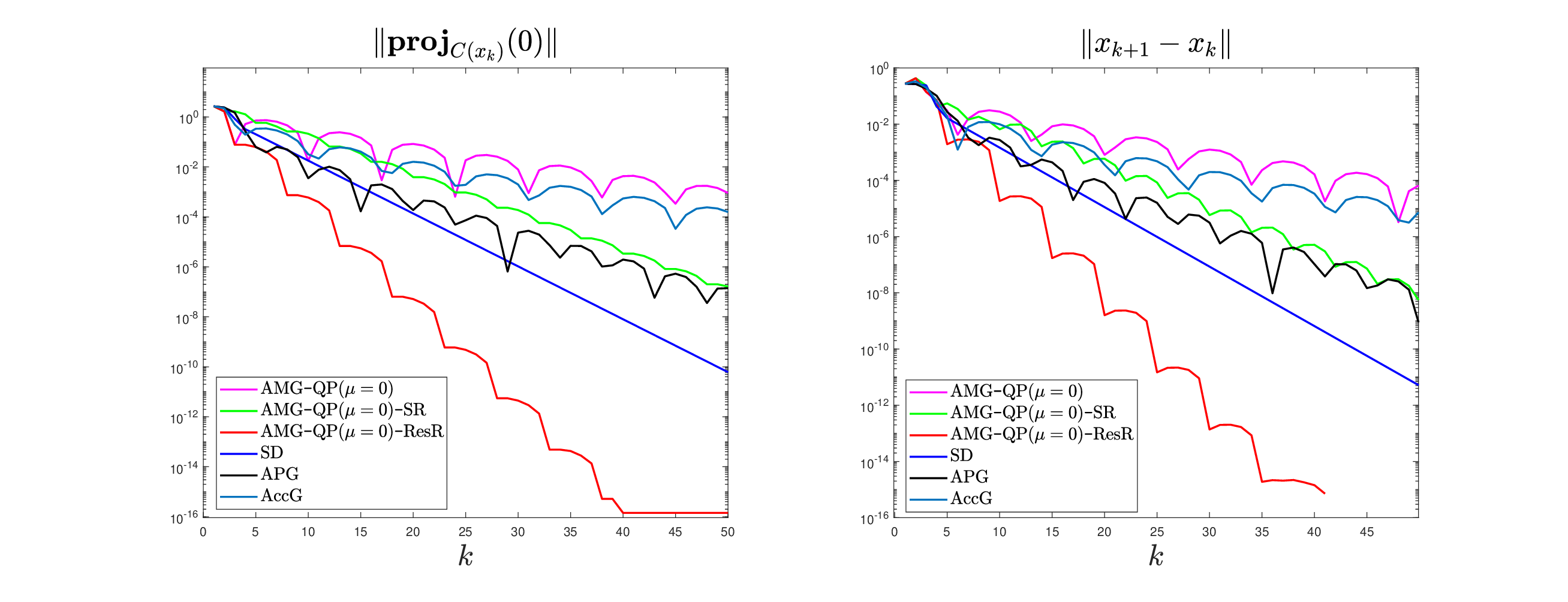} 
			\caption{The KKT residual $\|\proj_{C(x_k)}(0)\|$ and the iterate gap $\nm{x_{k+1}-x_k}$ $v.s.$ the iteration step $k$, with one initial sample point $x_0\in[-2,2]^n$ for the third example \cref{eq:ex3}.}
			\label{fig:test_NC_Resk_iter}
		\end{figure} 
		\begin{figure}[H]
			\centering 
			\includegraphics[width=0.78\textwidth]{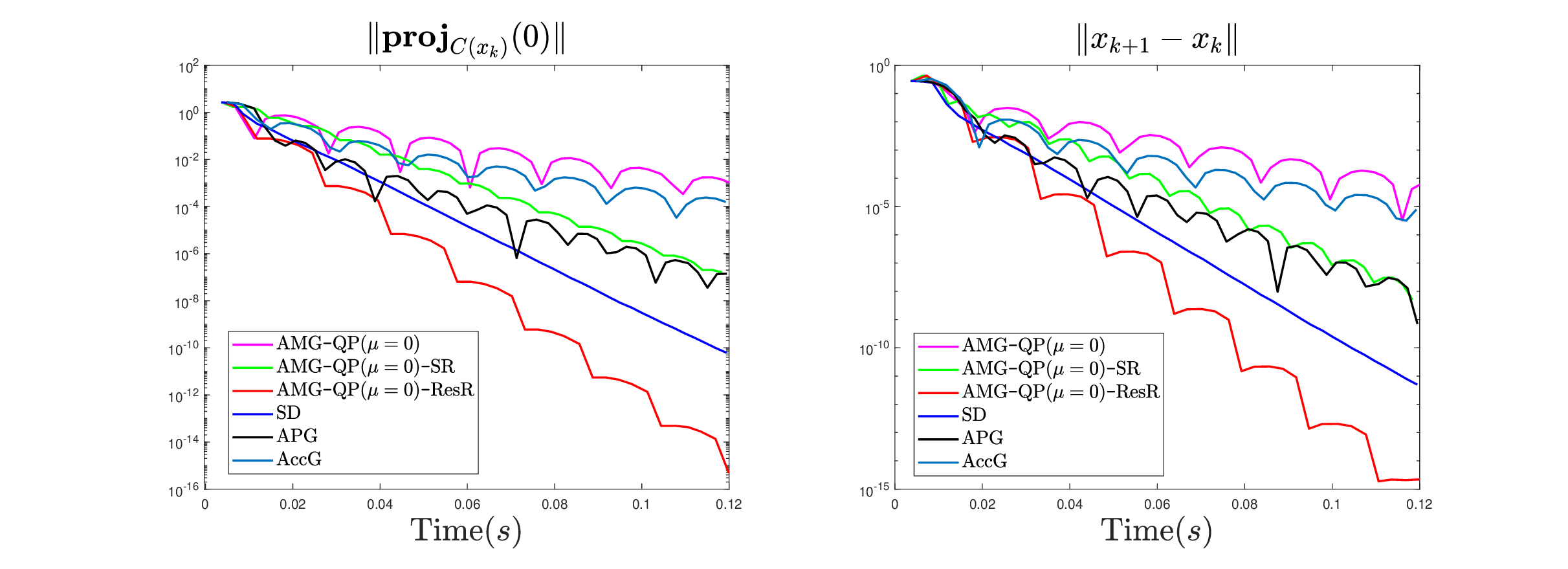} 
			\caption{The KKT residual $\|\proj_{C(x_k)}(0)\|$ and the iterate gap $\nm{x_{k+1}-x_k}$ $v.s.$ the iteration time $s$, with one initial sample point $x_0\in[-2,2]^n$ for the third example \cref{eq:ex3}.}
			\label{fig:test_NC_Resk_time}
		\end{figure}

		\begin{figure}[H]
			\centering 
			\includegraphics[width=0.66\textwidth]{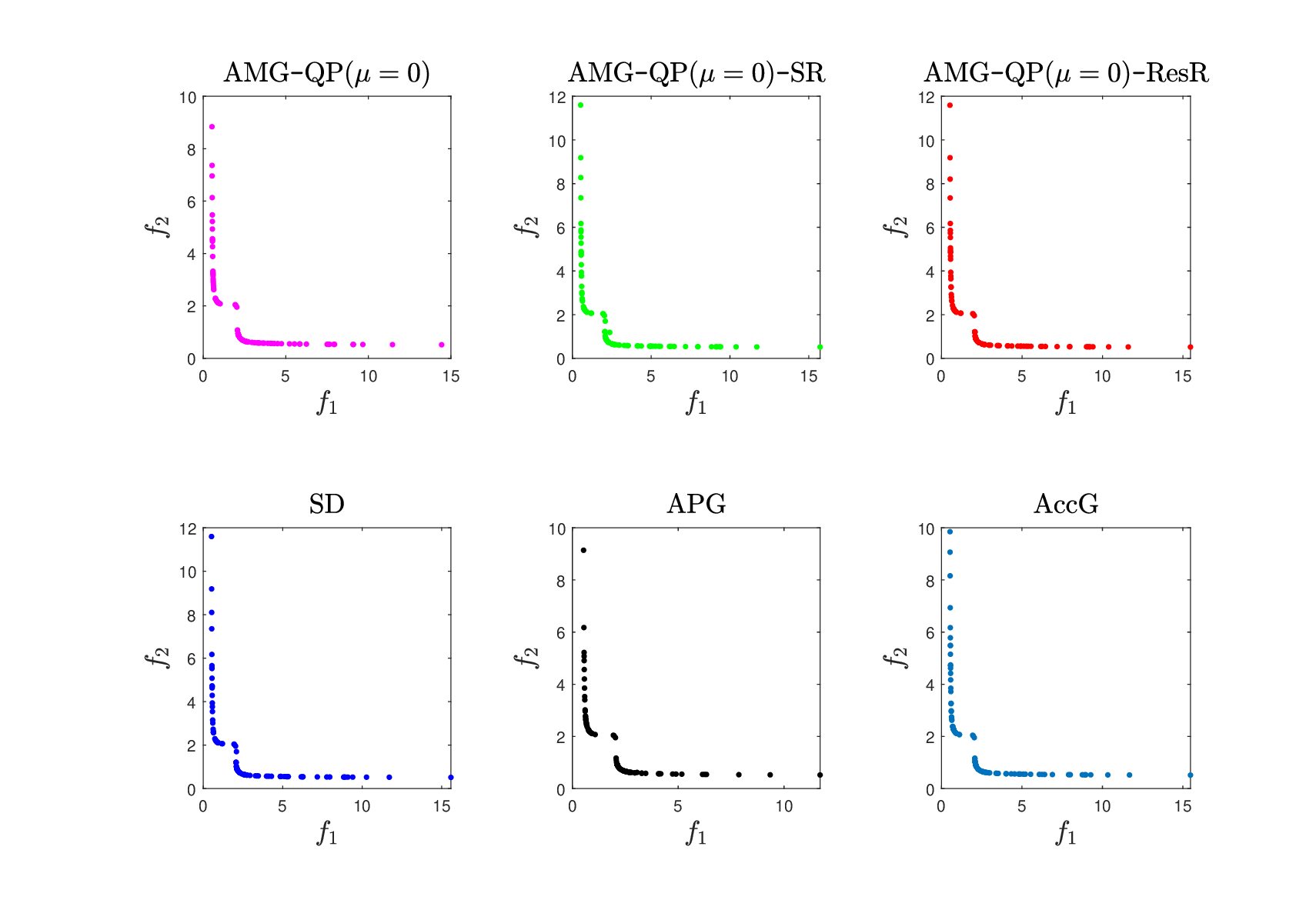} 
			\caption{The approximate Pareto front with $N=100$ initial sample points in $[-2,2]^n$ for the third example \cref{eq:ex3}.}
			\label{fig:test_NC_PF}
		\end{figure}

		\section{Conclusion}
		\label{sec:conclu}
		In this work, we present an accelerated gradient flow approach for convex multiobjective programming. The continuous time limit ODE of the multiobjective accelerated proximal gradient method \cite[Algorthm 2]{Tanabe2023a} is derived rigorously. To handle the convex case and the strongly convex case, we propose a novel accelerated multiobjective gradient (AMG) flow with adaptive time rescaling parameter and establish the exponential rate of a merit function. Then an accelerated multiobjective gradient method with fast rate $\min\{L/k^2,(1-\sqrt{\mu/L})^k\}$ is obtained from a proper implicit-explicit discretization of our continuous AMG flow. Moreover, we provide an effective adaptive residual restart technique to improve the practical performance of the proposed method. For future topics, we will focus on the nonsmooth (composite) case and the convergence analysis of the adaptive residual restart.
		
		\vskip.5cm 
		\noindent{\bf Acknowledgment} The authors would like to thank Dr. Jie Zhang (zjieabc@163.com) for her careful readings on an early version of this manuscript.

		\begin{appendices}
			\appendix
			\section{Note on a Subproblem}
			\label{sec:note-sub-prob}
			Let $C=\conv{p_1,\cdots,p_m}$ be given with $p_1,\cdots,p_m\in\R^n$. For $a>0,\,b\in[0,a]$ and $u,w\in\R^n$, define $x^\#\in\R^n$ implicitly by the nonlinear equation
			\begin{equation}\label{eq:non-eq-x}
				ax^\# = u-\proj_{C}(w-bx^\#).
			\end{equation}
			Let $P=[p_1,\cdots,p_m]$ and assume there exists some $\lambda^\#\in\Delta_m$ such that 
			\begin{equation}\label{eq:dual-rela}
				\proj_{C}(w-bx^\#) = \sum_{j=1}^{m}\lambda_j^\#p_j=P\lambda^\#.
			\end{equation}
			We then obtain from \cref{eq:non-eq-x,eq:dual-rela} that $	x^\#=
			(u-P\lambda^\#)/a$. According to the optimality condition of the projection in \cref{eq:non-eq-x}, we have
			\[
			\dual{P\lambda-P\lambda^\#,w-bx^\#-u+ax^\#}\leq 0\quad\forall\,\lambda\in\Delta_m,
			\]
			which gives 
			\[
			\dual{
				P^\top
				\left[(a-b)P\lambda^\#-aw+bu\right],\lambda-\lambda^\#}\geq 0\quad\forall\,\lambda\in\Delta_m.
			\]
			This actually corresponds to the optimality condition of the constrained optimization problem
			\[
			\lambda^\#\in\argmin_{\lambda\in\Delta_m}\,\Phi(\lambda),\quad\text{where }\Phi(\lambda) = \left\{
			\begin{aligned}
				{}&	\frac{1}{2}\nm{P\lambda-\frac{aw-bu}{a-b}}^2,&&\text{if}\,a>b,\\
				{}&\dual{P\lambda,u-w},&&\text{if}\,a=b.
			\end{aligned}
			\right.
			\]
			
			Consequently, the solution $x^\#$ to \cref{eq:non-eq-x} is characterized as below.
			\begin{thm}\label{thm:sub-prob-case1}
				Let $x^\#\in\R^n$ satisfy \cref{eq:non-eq-x}, then we have $x^\#=(u-v^\#)/a$ where
				\[
				\left\{
				\begin{aligned}
					{}&		v^\# = \proj_{C}\left(\frac{aw-bu}{a-b} \right),&&\text{if}\,a>b,\\
					{}&		v^\# \in\mathop{\argmin}\limits_{v\in C}\dual{v,u-w},&&\text{if}\,a=b.
				\end{aligned}
				\right.
				\]
			\end{thm}
			\begin{rem}
				Observing the fact that $	x+ 	\proj_{C}(-x) = \proj_{x+C}(0)$ for all $x\in\R^n$,
				the nonlinear equation \cref{eq:non-eq-x} is also equivalent to
				\[
				(a-b)x^\# = u - \proj_{C-w+bx^\#}(0).
				\]
				Starting from this formulation, the solution $x^\#$ under the two cases $a=b$ and $a>b$ has also been given in \cite[Lemmas A.1 and A.2]{Sonntag2024}. 
			\end{rem}	
		\end{appendices}
		
		\bibliographystyle{abbrv}

	\end{document}